\documentclass[12pt]{amsart}
\usepackage{amscd}
\usepackage{amsmath}
\usepackage{verbatim}
\usepackage[cmtip,arrow]{xy}
\usepackage{pb-diagram,pb-xy}
\usepackage{amssymb}
\usepackage{xypic}

\input cyracc.def

\numberwithin{equation}{section}
\renewcommand{\thesubsection}{\thesection\alph{subsection}}

\newtheorem{thm}[equation]{Theorem}
\newtheorem{theorem}[equation]{Theorem}
\newtheorem{prop}[equation]{Proposition}
\newtheorem{proposition}[equation]{Proposition}

\newtheorem{lemma}[equation]{Lemma}
\newtheorem{cor}[equation]{Corollary}
\newtheorem{corollary}[equation]{Corollary}

\theoremstyle{definition}
\newtheorem{rem}[equation]{Remark}
\newtheorem{example}[equation]{Example}
\newtheorem{dfn}[equation]{Definition}

\usepackage[hypertex]{hyperref}

\DeclareMathAlphabet{\cat}{OT1}{cmss}{m}{sl}

\title
{On standard norm varieties}

\keywords
{Norm varieties, Chow groups and motives, Steenrod operations.
{\em Mathematical Subject Classification (2010):}
14C25}

\author
{Nikita A. Karpenko}

\author
{Alexander S. Merkurjev}

\address
{
Institut de Math\'ematiques de Jussieu\\
Universit{\'e} Pierre et Marie Curie\\
Paris\\
FRANCE}

\email {karpenko {\it at} math.jussieu.fr}

\address
{Department of Mathematics\\
University of California\\
Los Angeles\\
CA\\
USA}

\email{merkurev {\it at} math.ucla.edu}

\date
{4 January 2012. Revised: 18 June 2012.}

\thanks
{Supported by the Max-Planck-Institut f\"ur Mathematik in Bonn}
\thanks{
The work of the second author has been supported by the
NSF grant DMS \#0652316}

\begin{document}

{
\newcommand{\s}{s}
\newcommand{\cM}{\mathcal{M}}
\newcommand{\Steen}{S}
\newcommand{\ISteen}{\mathbb{S}}
\newcommand{\sq}{\operatorname{sq}}
\newcommand{\st}{\operatorname{st}}

\newcommand{\onto}{\rightarrow\!\!\rightarrow}
\newcommand{\codim}{\operatorname{codim}}
\newcommand{\rdim}{\operatorname{rdim}}

\newcommand{\colim}{\operatorname{colim}}
\newcommand{\Div}{\operatorname{div}}
\newcommand{\Td}{\operatorname{Td}}
\newcommand{\td}{\operatorname{td}}
\newcommand{\itd}{\operatorname{itd}}

\newcommand{\ch}{\mathop{\mathrm{ch}}\nolimits}
\newcommand{\fch}{\mathop{\mathfrak{ch}}\nolimits}

\newcommand{\trdeg}{\operatorname{tr.deg}}

\newcommand{\SB}{X}

\newcommand{\Br}{\mathop{\mathrm{Br}}}
\newcommand{\Gal}{\mathop{\mathrm{Gal}}}
\renewcommand{\Im}{\mathop{\mathrm{Im}}}
\newcommand{\Pic}{\mathop{\mathrm{Pic}}}
\newcommand{\ind}{\mathop{\mathrm{ind}}}
\newcommand{\coind}{\mathop{\mathrm{coind}}}

\newcommand{\rk}{\mathop{\mathrm{rk}}}
\newcommand{\CH}{\mathop{\mathrm{CH}}\nolimits}

\newcommand{\BCH}{\mathop{\overline{\mathrm{CH}}}\nolimits}

\newcommand{\SO}{\operatorname{\mathrm{SO}}}
\newcommand{\OO}{\operatorname{\mathrm{O}}}

\newcommand{\Sp}{\operatorname{\mathrm{Sp}}}

\newcommand{\PGSp}{\operatorname{\mathrm{PGSp}}}

\newcommand{\Spin}{\operatorname{\mathrm{Spin}}}

\newcommand{\PGO}{\operatorname{\mathrm{PGO}}}

\newcommand{\PGL}{\operatorname{\mathrm{PGL}}}

\newcommand{\GL}{\operatorname{\mathrm{GL}}}

\newcommand{\SL}{\operatorname{\mathrm{SL}}}

\newcommand{\Ch}{\mathop{\mathrm{Ch}}\nolimits}
\newcommand{\BCh}{\mathop{\overline{\mathrm{Ch}}}\nolimits}

\newcommand{\TCh}{\mathop{\tilde{\mathrm{Ch}}}\nolimits}
\newcommand{\tCH}{\mathop{\mathrm{\widetilde{CH}}}\nolimits}

\newcommand{\IBCH}{\mathop{\bar{\mathrm{CH}}}\nolimits}
\newcommand{\BCHE}{\mathop{\mathrm{Che}}\nolimits}

\newcommand{\Stab}{\mathop{\mathrm{Stab}}\nolimits}

\newcommand{\res}{\mathop{\mathrm{res}}\nolimits}
\newcommand{\cores}{\mathop{\mathrm{cor}}\nolimits}

\newcommand{\cd}{\mathop{\mathrm{cd}}\nolimits}
\newcommand{\Gcd}{\mathop{\mathfrak{cd}}\nolimits}

\newcommand{\pr}{\operatorname{\mathit{pr}}}

\newcommand{\inc}{\operatorname{\mathit{in}}}

\newcommand{\mult}{\operatorname{mult}}

\newcommand{\Char}{\mathop{\mathrm{char}}\nolimits}
\newcommand{\Dim}{\mathop{\mathrm{Dim}}\nolimits}
\newcommand{\id}{\mathrm{id}}
\newcommand{\coker}{\mathrm{coker}}
\newcommand{\Z}{\mathbb{Z}}
\newcommand{\F}{\mathbb{F}}
\newcommand{\A}{\mathbb{A}}
\newcommand{\PP}{\mathbb{P}}
\newcommand{\Q}{\mathbb{Q}}
\newcommand{\BBF}{\mathbb{F}}
\newcommand{\HH}{\mathbb{H}}
\newcommand{\C}{\mathbb{C}}
\newcommand{\R}{\mathbb{R}}
\newcommand{\Gm}{\mathbb{G}_{\mathrm{m}}}
\newcommand{\Ga}{\mathbb{G}_{\mathrm{a}}}
\newcommand{\cF}{\mathcal F}
\newcommand{\cO}{\mathcal O}
\newcommand{\cE}{\mathcal E}
\newcommand{\cT}{\mathcal T}
\newcommand{\cA}{\mathcal A}
\newcommand{\cB}{\mathcal B}

\newcommand{\Mor}{\operatorname{Mor}}
\newcommand{\Spec}{\operatorname{Spec}}
\newcommand{\End}{\operatorname{End}}
\newcommand{\Hom}{\operatorname{Hom}}
\newcommand{\Aut}{\operatorname{Aut}}

\newcommand{\Tors}{\operatorname{Tors}}

\newcommand{\an}{0}
\newcommand{\op}{\mathrm{op}}

\newcommand{\X}{\mathfrak{X}}
\newcommand{\Y}{\mathcal{Y}}

\newcommand{\pt}{\mathbf{pt}}
\newcommand{\PS}{\mathbb{P}}

\newcommand{\type}{\mathop{\mathrm{type}}}

\newcommand{\Coprod}{\operatornamewithlimits{\textstyle\coprod}}
\newcommand{\Prod}{\operatornamewithlimits{\textstyle\prod}}
\newcommand{\Sum}{\operatornamewithlimits{\textstyle\sum}}
\newcommand{\Oplus}{\operatornamewithlimits{\textstyle\bigoplus}}

\newcommand{\disc}{\operatorname{disc}}

\newcommand{\<}{\left<}
\renewcommand{\>}{\right>}
\renewcommand{\ll}{\<\!\<}
\newcommand{\rr}{\>\!\>}

\newcommand{\stiso}{\overset{st}\sim}
\newcommand{\miso}{\overset{m}\sim}

\newcommand{\Dfn}{\stackrel{\mathrm{def}}{=}}

\marginparwidth 2.5cm
\newcommand{\Label}{\label}

\newcommand{\compose}{\circ}
\newcommand{\ur}{\mathrm{ur}}
\newcommand{\Ker}{\operatorname{Ker}}
\newcommand{\TCH}{\Tors\CH}
\newcommand{\CM}{\operatorname{CM}}
\newcommand{\BCM}{\operatorname{\overline{CM}}}

\newcommand{\Ann}{\operatorname{Ann}}

\newcommand{\corr}{\rightsquigarrow}

\newcommand{\IW}{\mathfrak{i}_0}
\newcommand{\iw}{\mathfrak{i}}
\newcommand{\jw}{\mathfrak{j}}

\newcommand{\Hight}{\mathfrak{h}}
\newcommand{\Height}{\mathfrak{h}}
\newcommand{\Dh}{\mathfrak{d}}

\newcommand{\IS}{i_S}

\newcommand{\Sym}{\operatorname{Sym}}
\newcommand{\D}{D}

\newcommand{\BC}{\ast}
\newcommand{\NBC}{\compose}
\newcommand{\ABC}{\star}
\newcommand{\NC}{\compose}
\newcommand{\AC}{\bullet}

\newcommand{\Fields}{\mathbf{Fields}}
\newcommand{\fgFields}{\mathbf{fgFields}}
\newcommand{\Bin}{\mathbf{2^0}}
\newcommand{\Sets}{\mathbf{pSets}}

\newcommand{\eps}{\varepsilon}

\renewcommand{\phi}{\varphi}

\newcommand{\RatM}{\dashrightarrow}
\newcommand{\Place}{\dashrightarrow}

\newcommand{\WR}{\mathcal{R}}

\newcommand{\mf}{\mathfrak}

\begin{abstract}
Let $p$ be a prime integer and $F$ a field of characteristic $0$.
Let $X$ be the {\em norm variety} of a symbol in
the Galois cohomology group $H^{n+1}(F,\mu_p^{\otimes n})$ (for some $n\geq1$), constructed in the proof of the
Bloch-Kato conjecture.
The main result of the paper
affirms that the function field $F(X)$ has the following property:
for any equidimensional variety $Y$, the change
of field homomorphism $\CH(Y)\to\CH(Y_{F(X)})$
of Chow groups with coefficients in integers localized at $p$
is surjective in codimensions $< (\dim X)/(p-1)$.
One of the main ingredients of the proof is a computation of Chow groups
of a (generalized) Rost motive (a variant of the main result not relying on this is given in
Appendix).
Another important ingredient is {\em $A$-triviality} of $X$,
the property saying that
the degree homomorphism on $\CH_0(X_L)$ is injective for any field extension $L/F$ with $X(L)\ne\emptyset$.
The proof involves
the theory of rational correspondences reviewed in Appendix.

\bigskip
\noindent
{\sc R\'esum\'e.}
Pour un nombre premier $p$ et un corps $F$ de caract\'eristique $0$,
soit $X$ la {\em vari\'et\'e de norme} d'un symbole dans le groupe de cohomology galoisienne
$H^{n+1}(F,\mu_p^{\otimes n})$ (avec certain $n\geq1$) construite en cours de d\'emonstration
de la conjecture de Bloch-Kato.
Le r\'esultat principal de cet article
affirme que le corps des fonctions $F(X)$ a la propri\'et\'e suivante:
pour toute vari\'et\'e \'equidimensionnelle $Y$, l'homomorphisme de changement de corps
$\CH(Y)\to\CH(Y_{F(X)})$ de groupes de Chow \`a coefficients entiers localis\'es en $p$ est surjectif en codimension
$< (\dim X)/(p-1)$.
Une des composantes principales de la preuve est le calcul de groupes de Chow du motif de Rost
g\'en\'eralis\'e
(un variant du r\'esultat principal ind\'ependant de ceci est propos\'e dans Appendix).
Un autre ingr\'edient important est la {\em $A$-trivialit\'e} de $X$, la propri\'et\'e
qui dit que
pour toute extension de corps $L/F$ avec $X(L)\ne\emptyset$,
l'homomorphisme de degr\'e pour $\CH_0(X_L)$ est injectif.
La preuve fait appara\^itre la th\'eorie de correspondances rationnelles revue dans Appendix.
\end{abstract}

\maketitle
\vspace{-3em}
\tableofcontents

\section
{Introduction}

Let $n$ be a positive integer and $p$ a prime integer.
A smooth complete geometrically irreducible variety $X$ over a field $F$
of characteristic $0$ is a \emph{$p$-generic splitting variety} for a symbol
$s\in H^{n+1}(F; \mu^{\otimes n}_p)$
if $s$ vanishes over a field extension $K/F$ if and only if $X$ over $K$ has a
closed point of degree prime to $p$.
A \emph{norm variety} of $s$ is a $p$-generic splitting variety of the smallest dimension $p^n-1$.
Norm varieties played an important role in the proof of the Bloch-Kato conjecture (see \cite{MR2811603}).

Let $Y$ be a smooth variety over $F$.
Write $\CH^i(Y)$ for the Chow group with coefficients in the integers
localized at $p$ and
$\widetilde{\CH^i}(Y)$ for the factor group of the Chow group $\CH^i(Y)$
modulo $p$-torsion elements and $p\CH^i(Y)$.
In \cite[Theorem 1.3]{MR2503240},
K.~Zainoulline proved, using the Landweber-Novikov operations in algebraic cobordism
theory, that every $Y$ and every norm variety $X$ of $s$ enjoy the following property:
if $i<(p^n-1)/(p-1)$, every class $\alpha$
in $\widetilde{\CH^i}(Y_{\bar F})$, where $\bar F$ is an algebraic closure of $F$, such that
$\alpha_{\bar F(X)}$ is $F(X)$-rational, is $F$-rational itself, i.e., $\alpha$ belongs to the image of the map
$\widetilde{\CH^i}(Y)\to \widetilde{\CH^i}(Y_{\bar F})$.
This statement is in the spirit of the Main Tool Lemma of A.~Vishik
\cite{MR2305424}.

In the present paper we improve this result by showing that every cycle in the Chow group $\CH^i(Y_{F(X)})$
is already defined over $F$, i.e, it comes from $\CH^i(Y)$.
More precisely, we prove the following theorem
(see Theorem \ref{mainZp} for a stronger statement and the proof):

\begin{theorem}
\label{intromain}
Let $F$ be a field of characteristic $0$ and
let $X$ be an  $A$-trivial $p$-generic splitting variety of a symbol
in $H^{n+1}(F, \mu^{\otimes n}_p)$.
Then the change of field homomorphism
$
\CH^i(Y )\to  \CH^i(Y_{F(X)})
$
is surjective
if $i< (p^n-1)/(p-1)$ for any equidimensional (not necessarily smooth) variety $Y$ over $F$.
Moreover, the bound $(p^n-1)/(p-1)$ is sharp.
\end{theorem}

The $A$-triviality property for $X$ means that the degree map $\deg:\CH_0(X_K)\to\Z_{(p)}$
is an isomorphism (i.e., the kernel $A(X_K)$ of the degree map is trivial)
for any field extension $K/F$ such that $X$ has a point over $K$.
(We believe that the $A$-triviality condition
should be also imposed in the statement of \cite[Theorem 1.3]{MR2503240}.)
Our proof of Theorem \ref{intromain} is ``elementary"
in the sense that it does not use the algebraic cobordism theory.
It is based on computation of Chow groups of the corresponding Rost
motive, an approach applied by A. Vishik in \cite[Remark on Page
665]{MR2804267} in order to obtain the conclusion of Theorem \ref{intromain}
for Pfister quadrics (with $p=2$).

In Section
\ref{section A-triviality of standard norm varieties}
we prove that the standard norm varieties (corresponding to nontrivial symbols)
constructed in \cite{MR2220090} are $A$-trivial,
so that Theorem \ref{intromain} can be applied to such
varieties.
In fact, we prove more (see Theorem \ref{chowmain} for a more explicit statement and the proof):

\begin{theorem}
Let $X$ be a standard norm variety of a nontrivial symbol over a field $F$ of characteristic $0$.
Then for any field extension $K/F$,
the degree map $\deg:\CH_0(X_K)\to \Z_{(p)}$ is injective.
\end{theorem}

In the proof we use the theory of rational correspondences developed by M.~Rost (unpublished)
and B.~Kahn/R.~Sujatha in \cite{KS}.
We review this theory in Appendix \ref{Rational correspondences}.
Another ingredient of the proof, a computation of Chow groups of Rost motives, is presented in Appendix
\ref{Chow groups of Rost motives}.
A variant of the main theorem valid in any characteristic $\ne p$ and involving the Steenrod operations
is given in Appendix \ref{Special correspondences}.

In Sections \ref{Abstract Rost motives} and \ref{Generic splitting varieties}
we develop a theory of (abstract) Rost motives.

\medskip
We use the following notation and conventions.
The base field $F$ is of arbitrary characteristic if not specified
otherwise (it is of characteristic $\ne p$, $p$ a fixed prime,
most of the time, and of characteristic $0$ in several places).
An \emph{$F$-variety} over a field $F$ is a separated scheme of finite type over $F$.

We fix a commutative unital ring $\Lambda$ and write $\CH=\CH_*$ for the Chow
group with coefficients in $\Lambda$.
For any integer $i$ and equidimensional $F$-variety $Y$,
we write $\CH^i(Y)$ for the Chow group $\CH_{\dim Y-i}(Y)$.

Many events in the paper happen in the \emph{category of Chow motives} (with coefficients in $\Lambda$,
see \cite{EKM}).
A \emph{Chow motive} is a pair $(X,\rho)$, where $X$ is a smooth complete variety over $F$ and
$\rho$ is a projector (idempotent) in the endomorphism ring of the motive $M(X)$ of $X$.
We say that a motive $M$ \emph{lives on $X$}, if $M\simeq(X,\rho)$ for some $\rho$ as above.

\bigskip
\noindent
{\sc Acknowledgements.}
The authors thank Markus Rost for useful comments and suggestions.

\section
{$A$-trivial varieties}

Let $X$ be a smooth complete irreducible $F$-variety.
Let $d$ be its dimension, and let $\rho$ be a fixed element of the Chow
group $\CH^d(X\times X)$ (considered as a correspondence $X\corr X$).

\begin{lemma}
\label{xi v xi}
The following two conditions on $\rho$ are equivalent:
\begin{enumerate}
\item
for any $F$-variety $Y$,
the image of any $\alpha\in\CH(X\times Y)$ under the pull-back to
$\CH(Y_{F(X)})$ coincides with the image of $\alpha\compose\rho$;
\item
$\rho_*[\xi]=[\xi]$, where $\xi$ is the generic point of $X$ and  $[\xi]$
is its class in the Chow group $\CH_0(X_{F(X)})$.
\end{enumerate}
\end{lemma}

\begin{proof}
$(2)\Rightarrow(1)$.
The image of $\alpha$ is equal to $\alpha_*[\xi]$.
In particular, the image of $\alpha\compose\rho$ is equal to
$(\alpha\compose\rho)_*[\xi]=\alpha_*(\rho_*[\xi])=\alpha_*[\xi]$ if
$\rho_*[\xi]=[\xi]$.

$(1)\Rightarrow(2)$.
Apply $(1)$ to $Y=X$ and the class of the diagonal of $X$ in place of $\alpha$.
\end{proof}

\begin{cor}
\label{cor xi v xi}
If $\rho$ satisfies the conditions of Lemma {\rm \ref{xi v xi}}, then for any
$F$-variety $Y$ the pull-back homomorphism $\CH(X\times Y)\compose\rho\to\CH(Y_{F(X)})$
is surjective.
\qed
\end{cor}

\begin{dfn}
A smooth complete $F$-variety $X$ is {\em $A$-trivial}, if for any field
extension $L/F$ with $X(L)\ne\emptyset$, the degree homomorphism
$\deg:\CH_0(X_L)\to\Lambda$ is an isomorphism.
\end{dfn}

\begin{rem}
\label{def-A-trivial}
The notion of $A$-triviality depends on $\Lambda$.
A variety $A$-trivial for $\Lambda=\Z$ is $A$-trivial for any $\Lambda$.
If $\Lambda\ne0$,
any $A$-trivial variety is geometrically irreducible.
\end{rem}

\begin{example}
\label{PHV}
Any projective homogeneous variety $X$ under an action of a semisimple
affine algebraic group is $A$-trivial. Indeed, if $X(L)\ne\emptyset$, the variety
$X_L$ is rational and therefore $\deg:\CH_0(X_L)\to\Lambda$ is an isomorphism by
Corollary \ref{bir inv}.
\end{example}

{\em Multiplicity} $\mult\rho$ of $\rho$ is the element of $\Lambda$
such that the push-forward of $\rho$ with respect to the first projection
$X\times X\to X$ is equal to $(\mult\rho)\cdot[X]$.

\begin{lemma}
\label{criter}
Assuming that
$X$ is $A$-trivial,
$\rho$ satisfies conditions of Lemma {\rm \ref{xi v xi}} if and only if
$\mult\rho=1$.
\end{lemma}

\begin{proof}
Since
$X$ is $A$-trivial,
the $0$-cycle classes $\rho_*[\xi],[\xi]\in\CH_0(X_{F(X)})$ coincide
if and only if their degrees coincide.
It remains to notice that $\deg[\xi]=1$ and $\deg\rho_*[\xi]=\mult\rho$.
\end{proof}

A trivial example of $\rho$ satisfying the conditions of Lemma \ref{xi v xi} is given by
the class of the diagonal of $X$.
Here is one more example:

\begin{example}
If $X$ is a projective homogeneous variety under an action of a semisimple
affine algebraic group and $\rho\in\CH^d(X\times X)$ is a projector
such that the summand $(X,\rho)$ of the Chow motive of $X$ is {\em upper} in the sense of
\cite[Definition 2.10]{upper},
then $\mult\rho=1$ and therefore $\rho$ satisfies conditions of Lemma \ref{xi v xi} by Lemma
\ref{criter} ($X$ is $A$-trivial by Example \ref{PHV}).
\end{example}

\begin{prop}
\label{general}
Assume that $\rho$ satisfies conditions of Lemma
{\rm \ref{xi v xi}}
(the assumption is satisfied, for instance, if $\mult\rho=1$ and
$X$ is $A$-trivial).
Also assume that $\rho$ is a projector.
Given an equidimensional $F$-variety $Y$ and an integer $m$
such that for any $i$ and any point $y\in Y$ of codimension $i$ the change of field homomorphism
$$
\rho^*\CH^{m-i}(X)\to\rho^*\CH^{m-i}(X_{F(y)})
$$
is surjective,
the change of field homomorphism
$$
\CH^m(Y)\to\CH^m(Y_{F(X)})
$$ is also surjective.
\end{prop}

\begin{proof}
Since $\rho^*(x)\times y=(x\times y)\compose\rho$ for any $x\in\CH(X)$,
$F$-variety $Y$ and $y\in\CH(Y)$ (where the composition of correspondences is taken in the sense of
\cite{MR2029161}, see also \cite[\S62]{EKM}),
the external product homomorphism
$\CH(X)\otimes\CH(Y)\to\CH(X\times Y)$
maps $\big(\rho^*\CH(X)\big)\otimes\CH(Y)$ to $\CH(X\times
Y)\compose\rho$.

Let us check that in our situation
the homomorphism
\begin{equation*}
\Oplus_{i}\big(\rho^*\CH^i(X)\big)\otimes_\Lambda\CH^{m-i}(Y)\to\CH^m(X\times Y)\compose\rho
\end{equation*}
is surjective.

Checking this, we may assume that $Y$ is integral and
proceed by induction on $\dim Y$ using
the exact sequence
$$
\Oplus_{Y'}\CH^{m-1}(X\times Y')\compose\rho\to\CH^m(X\times Y)\compose\rho\to\rho^*\CH^m(X_{F(Y)}),
$$
where the direct sum is taken over all integral subvarieties $Y'\subset Y$
of codimension $1$.
The sequence is exact because the sequence
$$
\Oplus_{Y'}\CH^{m-1}(X\times Y')\to\CH^m(X\times Y)\to\CH^m(X_{F(Y)})
$$
is exact and $\rho$ is a projector.

Now we consider the following commutative diagram
$$
\begin{CD}
\big(\rho^*\CH(X)\big)\otimes_\Lambda\CH(Y)@>>>\CH(X\times Y)\compose\rho\\
@VVV @VVV\\
\CH(Y)@>>>\CH(Y_{F(X)})
\end{CD}
$$
where the left homomorphism is induced by the augmentation map
$\CH(X)\to\Lambda$.
The right homomorphism is surjective by Corollary \ref{cor xi v xi}.
As we checked right above, the top homomorphism is surjective in codimension $m$.
Therefore the bottom homomorphism is also surjective in codimension $m$.
\end{proof}

The following statement is a particular case of \cite[Theorem 2.11 ($3\Rightarrow1$)]{MR2427051}:

\begin{lemma}
\label{surj}
Assume that
$X$ is $A$-trivial
and $1\in\deg\CH_0(X)$.
Then for any $F$-variety $Y$, the change of field homomorphism $\CH(Y)\to\CH(Y_{F(X)})$ is
an isomorphism.
\end{lemma}

\begin{proof}
To prove surjectivity, we note that
any $y\in\CH(Y_{F(X)})$ is the image of some $\alpha\in\CH(X\times Y)$.
If $x\in\CH_0(X)$ is an element of degree $1$, then
the correspondence $[X]\times x\in\CH^d(X\times X)$ satisfies by Lemma
\ref{criter} the conditions of Lemma \ref{xi v xi}.
Therefore
$\alpha\compose([X]\times x)\in\CH(X\times Y)$ is also mapped to $y\in\CH(Y_{F(X)})$.
On the other hand, $\alpha\compose([X]\times x)=[X]\times\alpha_*(x)$
is mapped to $\alpha_*(x)_{F(X)}$ and it follows that
$\alpha_*(x)$ is an element of $\CH(Y)$ mapped to $y$.

Injectivity follows by specialization (see \cite[\S20.3]{Fulton} or \cite{MR1418952}).
\end{proof}

\begin{cor}
\label{l}
Assume that $X$ is $A$-trivial.
Then for any $l\in\deg\CH_0(X)\subset\Lambda$ and any $F$-variety $Y$, the image of
$\CH(Y)\to\CH(Y_{F(X)})$ contains $l\CH(Y_{F(X)})$.
\end{cor}

\begin{proof}
It suffices to consider the case $l=\deg x$ for a closed point $x\in X$.
Let $L$ be the residue field of $x$.
The change of field homomorphism
$\CH(Y_L)\to\CH(Y_{L(X)})$ is surjective by Lemma \ref{surj},
and the transfer argument does the job.
\end{proof}

\section
{Abstract Rost motives}
\label{Abstract Rost motives}

In this section, the coefficient ring $\Lambda$
is $\Z_{(p)}$ (the ring of integers localized in a fixed
prime $p$) or $\F_p$ (the finite field of $p$ elements).

For any integer $n\geq1$, an {\em abstract Rost motive of degree $n+1$} with coefficients in $\Lambda$ is
a Chow motive $\WR$ with coefficients in $\Lambda$
living on a smooth complete geometrically irreducible variety $X$ such that
for any field extension $L/F$ with
$1\in\deg\CH_0(X_L)$ one has
$\WR_L\simeq \Lambda\oplus\Lambda(b)\oplus\dots\oplus\Lambda((p-1)b)$,
where $b:=(p^n-1)/(p-1)$.

In particular, $\dim X\geq p^n-1=(p-1)b$.
Pulling-back the projector of $\WR$ with respect to the diagonal of $X$,
produces a $0$-cycle class of degree $p$
(cf.  \cite[Lemma 2.21]{upper})
showing that
$\deg\CH_0(X)\supset p\Lambda$.
It follows that the the ideal $\deg\CH_0(X)\subset\Lambda$ of the coefficient ring $\Lambda$
is equal either to $p\Lambda$ or to $\Lambda$.

The condition $1\in\deg\CH_0(X_L)$ appearing in the definition means the
same for $\Lambda=\Z_{(p)}$ as for $\Lambda=\F_p$.
In $\Lambda$-free terms, it means that the variety $X_L$ has a closed
point of a prime to $p$ degree.

By the very definition, the multiplicity of the projector of $\WR$ (which we call an {\em abstract Rost projector})
is equal to $1$.

Note that for any field extension $L/F$
such that
$1\in\deg\CH_0(X_L)$ we have
$$
\End \WR_L=\End\Lambda\times\End\Lambda(b)\times\dots\times\End\Lambda((p-1)b)=\Lambda^p.
$$
In particular, $\End\WR_L\to\End\WR_{L'}$ is an isomorphism for any field
extension $L'/L$.

We fix an integer $n\geq1$ and consider only abstract Rost motives of
degree $n+1$ below.

\begin{lemma}
\label{R indec}
An abstract Rost motive is indecomposable if (and only if) $1\not\in\deg\CH_0(X)$.
In particular, an abstract Rost motive with coefficients in $\Z_{(p)}$ is
indecomposable if and only if the corresponding abstract Rost motive with
coefficients in $\F_p$ is indecomposable.
\end{lemma}

\begin{proof}
Assuming that
$\WR=R_1\oplus R_2$ with $R_1,R_2\ne0$, we get $(R_1)_{F(X)}, (R_2)_{F(X)}\ne0$
by the nilpotence principle \cite[Proposition 3.1]{MR2393083}.
It follows by the Krull-Schmidt principle of \cite{MR2264459} that
$(R_1)_{F(X)}$ is isomorphic to a direct sum of shifts of $m$ copies of
$\Lambda$ where $0<m<p$.
Pulling back the projector of $R_1$ via the diagonal of $X$, we get a $0$-cycle of degree $m\in\Lambda$
(cf.  \cite[Lemma 2.21]{upper}).
This contradicts $1\not\in\deg\CH_0(X)$.
\end{proof}

\begin{lemma}
\label{pEnd}
For any abstract Rost motive $\WR$ one has
$$
p\End \WR_{F(X)}\subset \Im(\End \WR\to\End \WR_{F(X)}).
$$
\end{lemma}

\begin{proof}
The statement being vacuous for $\Lambda=\F_p$, one may assume that $\Lambda=\Z_{(p)}$ in the proof.
We also may assume that $\WR$ is indecomposable.

Let $L$ be the residue field of a closed point on $X$ of degree not
divisible by $p^2$ (but, of course, divisible by $p$).
Since $X(L)\ne\emptyset$,
$\End \WR_L\to \End \WR_{L(X)}$ is an isomorphism.
For any $\alpha\in\End \WR_{F(X)}$ the endomorphism
$p\alpha$ is
in the image of the coinciding with multiplication by $[L:F]$ composition
$\End \WR_{F(X)}\to\End \WR_{L(X)}\to\End \WR_{F(X)}$ and
therefore in the image of the composition
$\End \WR_L\to\End \WR_{L(X)}\to\End \WR_{F(X)}$
which coincides with the composition
$\End \WR_L\to\End \WR\to\End \WR_{F(X)}$.
\end{proof}

\begin{lemma}
\label{mult1end}
Any multiplicity $1$ endomorphism of an indecomposable abstract Rost motive is an automorphism.
\end{lemma}

\begin{proof}
We take some $\alpha\in\End \WR$ of multiplicity $1$.
Since the ring
$
\End \WR_{F(X)}
$
is the product of $p$ copies of $\Lambda$,
the endomorphism $\alpha_{F(X)}\in\End \WR_{F(X)}$ is given by a $p$-tuple
of elements in $\Lambda$.
This $p$-tuple starts with $1$ (because the starting component of the $p$-tuple is the multiplicity of $\alpha$).
Actually, every component of the $p$-tuple is congruent to $1$ modulo $p$.
Indeed, if a component of $\alpha_{F(X)}$ is $\lambda\not\equiv1$, the
{\em $F$-rational} (i.e., coming from $F$) endomorphism $(\alpha-\lambda\cdot\id)_{F(X)}$
considered in  $(\End \WR_{F(X)})_{\Lambda}\otimes\F_p$ has a nontrivial and a trivial component.
Rasing to ($p-1$)th power, provides us with a nontrivial $F$-rational
idempotent.
By the nilpotence principle mentioned in the proof of Lemma \ref{R indec}, this produces a nontrivial idempotent in the ring
$(\End \WR)_\Lambda\otimes\F_p$, contradicting Lemma \ref{R indec}.

So, every component of $\alpha_{F(X)}$ is congruent to $1$ modulo $p$
(and this is the end of the proof in the case of $\Lambda=\F_p$).
In particular, $\alpha_{F(X)}$ is invertible.
By Lemma \ref{pEnd}, the inverse of $\alpha_{F(X)}$ is rational
(because each component of the inverse is also congruent to $1$ modulo $p$).
Therefore we may assume that $\alpha_{F(X)}=1$.
In this case, $\alpha-1$ is nilpotent by nilpotence principle applied one more
time, and it follows that $\alpha$ itself is invertible.
\end{proof}

It turns out that $\WR$ is determined by the class of $X$ with
respect to the following equivalence relation:
$X\sim X'$ if there exist multiplicity $1$ correspondences $X\corr X'$ and $X'\corr X$.
(In slightly different terms, $X\sim X'$ means that for any field extension $L/F$ one has $1\in\deg\CH_0(X_L)$
if and only if $1\in\deg\CH_0(X'_L)$.
Note that the equivalence relations for $\Lambda=\Z_{(p)}$ and $\Lambda=\F_p$ coincide.)
More precisely, we have

\begin{prop}
\label{R and R}
Abstract Rost motives $\WR$ and $\WR'$ living on
varieties $X$ and $X'$ are isomorphic if and only if the varieties are equivalent.
If an indecomposable abstract Rost motive lives on one of two equivalent varieties, then
it also lives on the other one.
\end{prop}

\begin{proof}
Mutually inverse isomorphisms between $\WR$ and $\WR'$ living on varieties $X$ and
$X'$ are given by some correspondences $f:X\corr X'$ and $g:X'\corr X$.
Since $\WR=(X,g\compose f)$, we have $\mult(g\compose f)=1$.
As $\mult(g\compose f)=\mult(g)\cdot\mult(f)$, the correspondences $f$
and $g$ have prime to $p$ multiplicities showing that $X\sim X'$.

Now given an indecomposable $\WR$ living on $X$ and given some $X'$ equivalent to
$X$, we show that $\WR$ is a direct summand of the motive $M(X')$ of $X'$.
The equivalence $X\sim X'$ provides us with multiplicity $1$ correspondences
$f:X\corr X'$ and $g:X'\corr X$.
The composition $g\compose f$ considered on $\WR$ is a multiplicity $1$
endomorphism of $\WR$.
This endomorphism is an automorphism by Lemma \ref{mult1end}.

Finally, if we are given some $\WR$ and $\WR'$ living on some equivalent $X$
and $X'$, and we want to show that $\WR\simeq\WR'$, then we may assume that
$\WR$ and $\WR'$ are indecomposable
and consider morphisms $\WR\to \WR'$ and $\WR'\to \WR$ given by multiplicity $1$ correspondences
$f:X\corr X'$ and $g:X'\corr X$.
Repeating the above argument, we show that $\WR$ is a direct summand of
$\WR'$.
Therefore $\WR\simeq \WR'$ by indecomposability of $\WR'$.
\end{proof}

\begin{cor}
Abstract Rost motives with coefficients in $\Z_{(p)}$, becoming isomorphic
after the change of coefficients $\Z_{(p)}\to\F_p$, are isomorphic.
\qed
\end{cor}

We recall that {\em canonical $p$-dimension} $\cd_p X$ of a smooth complete
irreducible variety $X$ is the least dimension of a closed subvariety
$Y\subset X$ possessing a multiplicity $1$ correspondence $X\corr Y$,
cf. \cite{MR2258262}.
One always has $\cd_p X\leq\dim X$, and $X$ is called {\em
$p$-incompressible} in the case of equality.
Canonical $p$-dimensions of equivalent varieties coincide:

\begin{lemma}
\label{cdp=cdp}
If $X\sim X'$, then $\cd_p X=\cd_p X'$.
\end{lemma}

\begin{proof}
Assuming that $X\sim X'$, it suffices to show that $\cd_p X\leq\cd_p X'$.
Let $Y'$ be a closed irreducible subvariety of $X'$ with a multiplicity $1$
correspondence $X'\corr Y'$ and with $\dim Y'=\cd_p X'$.
Then there exists a prime correspondence $X'\corr Y'$ of prime to $p$
multiplicity.
Such a correspondence is given by an irreducible closed subvariety
$Z'$ in $X'\times Y'$ such that the projection $Z'\to X'$ is surjective
and the field extension $F(X')\hookrightarrow F(Z')$ is of finite
prime to $p$ degree.
By minimality of $Y'$, the projection $Z'\to Y'$ is also surjective.

The variety $X'$ having an $F(Y')$-point, the variety
$X_{F(Y')}$ has a $0$-cycle of degree $1$.
Consequently, there exists a closed subvariety $Z\subset Y'\times X$
surjective over $Y'$ with $F(Y')\hookrightarrow F(Z)$ of finite prime to $p$
degree.
Let $Y\subset X$ be the image of the projection $Z\to X$.

We have obtained a diagram of fields, shown below on the left, in which the vertical
embeddings are of finite prime to $p$ degrees.
By \cite[Lemma 3.1]{MR2258262}, it can be completed to a commutative
diagram of fields, shown on the right, in which the vertical
embeddings are still of finite prime to $p$ degrees:
{\Small
$$
\begin{CD}
F(Y)@>>>F(Z) @. \\
@. @AAA @. \\
@. F(Y') @>>> F(Z')\\
@. @. @AAA \\
@. @. F(X')
\end{CD}
\hspace{12ex}
\begin{CD}
F(Y)@>>>F(Z) @>>> L \\
@. @AAA @AAA \\
@. F(Y') @>>> F(Z')\\
@. @. @AAA \\
@. @. F(X')
\end{CD}
$$
}
Taking a model $U$ for $L$ and considering the class of the closure of the image of the
induced rational map $U\to X'\times Y$, we get a prime correspondence $X'\corr Y$ of a prime to $p$ multiplicity,
showing that there exists a
multiplicity $1$ correspondence $X'\corr Y$.
Composing it with a multiplicity $1$ correspondence $X\corr X'$ (the composition is defined because $X'$ is
smooth complete), we get a
multiplicity $1$ correspondence $X\corr Y$ showing that $\cd_p X\leq\dim Y\leq \dim Z=\dim Y'=\cd_p
X'$.
\end{proof}

\begin{rem}
Lemma \ref{cdp=cdp} is easier to prove out of the (equivalent) definition
of canonical $p$-dimension of a smooth complete variety $X$ as the
{\em essential $p$-dimension} of the class of fields $L/F$ with
$X(L)\ne\emptyset$, given in \cite[\S1.6]{ASM-ed}.
Indeed, enlarging the above class of fields to the class
of fields $L/F$ with $1\in\deg\CH_0(X_L)$ keeps its essential $p$-dimension.
And, as we already mentioned, such enlarged classes of fields given by equivalent varieties coincide.
\end{rem}

\begin{lemma}
\label{cdp}
If an indecomposable abstract Rost motive (of degree $n+1$) lives on a variety $X$,
then $\cd_p X\geq p^n-1$.
\end{lemma}

\begin{proof}
For a closed subvariety $Y\subset X$ with a multiplicity $1$
correspondence $X\corr Y$, we consider the endomorphism $\alpha\in\End \WR$
given by the composition of correspondences
$X\corr Y\hookrightarrow X$.
(More explicitly, $\alpha$ is the composition $X\corr X\corr Y\hookrightarrow X\corr X$ with
$X\corr X$ being the projector of $R$.)
Since $\mult\alpha=1$, $\alpha$ is invertible by Lemma \ref{mult1end}.
On the other hand, the very last component of $\alpha_{F(X)}$ can be nonzero only if $\dim Y\geq
p^n-1$.
Indeed, this last component is given by the action of $\alpha_{F(X)}^*$ on
$\CH^{p^n-1}(\WR_{F(X)})\subset\CH^{p^n-1}(X_{F(X)})$, but since $\alpha$
lies in the image of the push-forward $\CH(X\times Y)\to\CH(X\times X)$, the action on the
whole group $\CH^{p^n-1}(X_{F(X)})$ is trivial if $\dim Y<p^n-1$.
\end{proof}

\begin{cor}
\label{aRinc}
If an indecomposable abstract Rost motive lives on a variety $X$ of dimension $p^n-1$,
then $X$ is $p$-incompressible.
\qed
\end{cor}

The following Lemma is inspired by \cite[Lemma 9.3]{markus}:

\begin{lemma}
\label{ura}
If $X$ is $p$-incompressible, then for any $i>0$ and any $\alpha\in\CH^i(X)$ and
$\beta\in\CH_i(X_{F(X)})$,
the degree of the $0$-cycle class given by the
product $\alpha_{F(X)}\cdot\beta\in\CH_0(X_{F(X)})$ is divisible by $p$.
\end{lemma}

\begin{proof}
If $\deg(\alpha_{F(X)}\cdot\beta)$ is not divisible by $p$ for some
$\alpha\in\CH^i(X)$ and $\beta\in\CH_i(X_{F(X)})$ with positive $i$,
we can find a closed irreducible subvariety $Y\subset X$
of codimension $i$
with $\deg([Y]_{F(X)}\cdot\beta)$ not divisible by $p$.
Since the product $[Y]_{F(X)}\cdot\beta$ is represented by a $0$-cycle
class on $Y_{F(X)}$, there exists a multiplicity $1$ correspondence
$X\corr Y$ showing that $\cd_p X\leq \dim X-i<\dim X$  contradicting the assumption that
$X$ is $p$-incompressible.
\end{proof}

\begin{cor}[{cf. \cite[Lemma 9.3]{markus}}]
\label{cor3.11}
If an abstract Rost motive lives on a variety $X$ of dimension $p^n-1$ and such that
$1\not\in\deg\CH_0(X)$,
then for any $i>0$ and any $\alpha\in\CH^i(X)$ and
$\beta\in\CH_i(X_{F(X)})$,
the degree of the product $\alpha_{F(X)}\cdot\beta\in\CH_0(X_{F(X)})$ is
divisible by $p$.
\qed
\end{cor}

\section
{Generic splitting varieties}
\label{Generic splitting varieties}

In this section $\Lambda$ is $\Z_{(p)}$ or $\F_p$ and
the base field $F$ is of characteristic $\ne p$ if not specified otherwise.

For $n\geq1$, an element $s\in H^{n+1}(F,\mu_p^{\otimes n})$ is a {\em symbol}, if
it is equal to the cup product of an element of $H^1(F,\Z/p\Z)$ and $n$
elements of $H^1(F,\mu_p)$.
A smooth complete geometrically irreducible $F$-variety $X$ is a {\em $p$-generic splitting variety}
of a symbol $s$, if for any field extension $L/F$ one has $s_L=0$ if and
only if $1\in\deg\CH_0(X_L)$ (it is a {\em generic splitting variety}
of $s$, if $s_L=0$ $\Leftrightarrow$ $X(L)\ne\emptyset$).

Clearly, given a symbol $s$ and a $p$-generic splitting variety $X$ of $s$,
a smooth complete geometrically irreducible variety $X'$
is also a $p$-generic splitting variety of $s$ if and only if $X\sim X'$.

A symbol $s'$ is {\em similar} to $s$, if $s'=as$ for a nonzero
$a\in\Z/p\Z$.
Similar symbols vanish over precisely the same fields so that
$p$-generic splitting varieties of similar symbols are equivalent.

According to
\cite{MR2220090},
in characteristic $0$,
for any symbol $s$, there exists a $p$-generic splitting variety of
dimension $p^n-1$.
The construction of such varieties
is recalled in Section
\ref{Subsection $A$-triviality of standard norm varieties}.

An abstract Rost motive $\WR=\WR_s$ living on a $p$-generic splitting variety of a
symbol $s\in H^{n+1}(F,\mu_p^{\otimes n})$ is called a {\em Rost motive}
of the symbol.

\begin{thm}
Assume that $\Char F=0$.
For any symbol $s\in H^{n+1}(F,\mu_p^{\otimes n})$,
a Rost motive $\WR_s$ of $s$ exists.
Moreover,
the isomorphism class of $\WR_s$ determines and is determined by the similarity class of $s$.
\end{thm}

\begin{proof}
The existence statement is proved in \cite{MR2811603} and \cite{MR2220090}.
By Proposition \ref{R and R}, the isomorphism class of $\WR_s$ is
determined by the similarity class of $s$.
Finally, if $\WR_{s'}\simeq\WR_s$ for a symbol $s'$,
the symbols $s$ and $s'$ vanish over precisely the same field extensions of $F$
and therefore are similar by \cite[Theorem 2.1]{MR2645334}.
\end{proof}

\begin{rem}
In characteristic $0$
one may show using
Theorem \ref{main} that for any Rost motive $\WR$ of a nonzero symbol
living on a variety $X$, the homomorphism of $\Z_{(p)}$-algebras
$$
\End\WR\to\End\WR_{F(X)}=(\Z_{(p)})^p
$$
is injective and has as image the unital $\Z_{(p)}$-subalgebra of $(\Z_{(p)})^p$ generated by
$p(\Z_{(p)})^p$.
This explains Lemmas \ref{pEnd} and \ref{mult1end}.
\end{rem}

Keeping the characteristic $0$ assumption,
it follows by Proposition \ref{R and R} that any $p$-generic splitting
variety of $s$ admits a Rost motive and the isomorphism class of a Rost motive on such a
variety only depends on $s$.
It follows also that $p^n-1$ is the least dimension of a $p$-generic splitting variety of a
symbol.
The $p$-generic splitting varieties of dimension $p^n-1$ are called {\em norm
varieties}.

\begin{thm}[Version for $\Lambda=\Z_{(p)}$ and $\Lambda=\F_p$]
\label{mainZp}
Let $F$ be a field of characteristic $0$.
Given an $A$-trivial $p$-generic splitting variety $X$ of a symbol
$s\in H^{n+1}(F,\mu_p^{\otimes n})$,
the change of field homomorphism $\CH(Y)\to\CH(Y_{F(X)})$
is surjective in codimensions $< (p^n-1)/(p-1)$
for any equidimensional variety $Y$.
It is also surjective in codimension $=(p^n-1)/(p-1)$ for a given $Y$ provided that
$s_{F(\zeta)}\ne0$
for each generic point $\zeta\in Y$.
\end{thm}

\begin{proof}
If $s=0$, then $1\in\deg\CH_0(X)$ and the statement of Theorem \ref{mainZp} is a particular case
of Lemma \ref{surj}.
Below in the proof we are assuming that $s\ne0$.

Let $\rho$ be a projector on $X$ giving the Rost motive.
By Proposition \ref{general}, to prove the first statement of Theorem \ref{mainZp},
it suffices to check that
for any field extension $L/F$ the change of field homomorphism
$$
\rho^*\CH(X)\to\rho^*\CH(X_L)
$$
is surjective in codimension $<m:=(p^n-1)/(p-1)$.
This condition is satisfied by Theorem \ref{main}.

To prove the second statement of Theorem \ref{mainZp}, it suffices to
additionally check that
for any generic point $\zeta\in Y$ the change of field homomorphism
$$
\rho^*\CH^m(X)\to\rho^*\CH^m(X_{F(\zeta)})
$$
is surjective.
Since
$s_{F(\zeta)}\ne0$,
this condition is satisfied by Theorem \ref{main}
as well.
\end{proof}

Our main example of $X$ for which Theorem \ref{mainZp} can be applied is
given by the standard norm variety of a symbol in $H^{n+1}(F,\mu_p^{\otimes
n})$, constructed in Section \ref{Subsection $A$-triviality of standard norm varieties}.
Such a variety is $A$-trivial by Theorem \ref{chowmain}.

The standard norm variety $X$ of a nonzero symbol also provides an example showing that
the boundary $b:=(p^n-1)/(p-1)$ of the first (and main) statement of Theorem
\ref{mainZp} is sharp.
Indeed, the element $H\in\CH^b(X_{F(X)})$ considered in Section \ref{Special
correspondences}, does not come from $F$.

A construction similar to \cite[Proof of Theorem
3.4]{MR2804267}, proves

\begin{cor}
\label{alaVishik}
For any field $F$ of characteristic $0$, any prime $p$ and any integer
$n\geq1$,
there exists a field extension $F'/F$ such that
$H^{n+1}(F',\mu_p^{\otimes n})=0$ and $\CH(Y)\to\CH(Y_{F'})$ is surjective in codimensions $< (p^n-1)/(p-1)$
for any equidimensional $F$-variety $Y$.
\qed
\end{cor}

As indicated in \cite[Remark after Theorem 3.4]{MR2804267}, Corollary
\ref{alaVishik} shows that ``modulo $p$ and degree $>n$ cohomological invariants of
equidimensional algebraic varieties could not affect rationality of cycles of codimension
$<(p^n-1)/(p-1)$''.

\begin{thm}[Version for $\Lambda=\Z$]
\label{mainZ}
We have $\Lambda=\Z$ in this statement.
Let $F$ be a field of characteristic $0$.
Given an $A$-trivial $p$-generic splitting variety $X$  of a symbol
$s\in H^{n+1}(F,\mu_p^{\otimes n})$ such that $p\in\deg\CH_0(X)$,
the homomorphism $\CH(Y)\to\CH(Y_{F(X)})$
is surjective in codimensions $< (p^n-1)/(p-1)$
for any equidimensional variety $Y$.
It is also surjective in codimension $=(p^n-1)/(p-1)$ for a given $Y$ provided that
$s_{F(\zeta)}\ne0$
for each generic point $\zeta\in Y$.
\end{thm}

\begin{proof}
By Theorem \ref{mainZp}, $\CH(Y)\to\CH(Y_{F(X)})$ is
surjective modulo $p$ in the codimensions considered.
Since $p\in\deg\CH_0(X)$, the image of
$\CH(Y)\to\CH(Y_{F(X)})$ contains $p\CH(Y_{F(X)})$ by Corollary \ref{l}.
\end{proof}

\begin{example}
For $p=2$, let $a_0,\dots,a_n\in F^\times$, $s$ the symbol
$s:=(a_0)\cup\dots\cup(a_n)$, and $X$ the norm quadric
$\<-a_0\>\bot\ll a_1,\dots,a_n\rr=0$.
Then $X$ is an $A$-trivial (see Example \ref{PHV}) norm variety of $s$ with $2\in\deg\CH_0(X)$ so
that Theorem \ref{mainZ} applies.
The result obtained has been originally established by A. Vishik in \cite[Corollary 3.3]{MR2804267}.
Using \cite{MR2015051} in place of Appendix \ref{Chow groups of Rost
motives},
the characteristic $0$ assumption can be replaced by characteristic $\ne2$
assumption in this statement.
\end{example}

\begin{example}
Let $\Char F=0$, $p=3$, and
let $X$ be any (among $15$) nontrivial projective homogeneous $F$-variety under an action of a
given absolutely simple affine algebraic group of type $F_4$ over $F$.
Then the conclusion of Theorem \ref{mainZp} holds for $X$ with $n=2$.
Indeed, $X$ is $A$-trivial (for any coefficient ring) by Example \ref{PHV}.
The modulo $3$ portion of the {\em Rost invariant} (see \cite{MR1999385})
provides us with an element $s\in H^3(F,\mu_3^{\otimes 2})$.
This element is a symbol by \cite[p. 303]{MR1691951} (see also \cite[p.
21]{MR2528487}).
Moreover, $X$ is a $3$-generic splitting variety of $s$ (see \cite[\S15.5]{MR2528487}),
so that Theorem \ref{mainZp} applies.
The result obtained is
an enhancement (in several respects) of
\cite[Case $p=3$ of Corollary 1.4)]{MR2503240}.

If $F$ has no finite extensions of degree prime to $3$, we have
$3\in\deg\CH_0(X)$ for $\Lambda=\Z$, and Theorem \ref{mainZ} applies.
\end{example}

}

{

\renewcommand{\(}{\bigl(}
\renewcommand{\)}{\bigr)}


\newcommand{\m}{^{\times}}
\newcommand{\ff}{F^{\times}}
\newcommand{\fs}{F^{\times 2}}
\newcommand{\llg}{\longrightarrow}
\newcommand{\tens}{\otimes}
\newcommand{\inv}{^{-1}}


\newcommand{\Dfn}{\stackrel{\mathrm{def}}{=}}
\newcommand{\iso}{\stackrel{\sim}{\to}}


\newcommand{\mult}{\mathrm{mult}}
\newcommand{\sep}{\mathrm{sep}}
\newcommand{\id}{\mathrm{id}}
\newcommand{\diag}{\mathrm{diag}}
\newcommand{\op}{^{\mathrm{op}}}


\newcommand{\ihom}{\mathcal{H}om}
\newcommand{\CH}{\operatorname{CH}}
\newcommand{\F}{\operatorname{F}}
\renewcommand{\Im}{\operatorname{Im}}
\newcommand{\Ad}{\operatorname{Ad}}
\newcommand{\NN}{\operatorname{N}}
\newcommand{\Ker}{\operatorname{Ker}}
\newcommand{\Pic}{\operatorname{Pic}}
\newcommand{\Tor}{\operatorname{Tor}}
\newcommand{\Lie}{\operatorname{Lie}}
\newcommand{\ind}{\operatorname{ind}}
\newcommand{\ch}{\operatorname{char}}
\newcommand{\Inv}{\operatorname{Inv}}
\newcommand{\coker}{\operatorname{Coker}}
\newcommand{\res}{\operatorname{res}}
\newcommand{\Br}{\operatorname{Br}}
\newcommand{\Spec}{\operatorname{Spec}}
\newcommand{\SK}{\operatorname{SK}}
\newcommand{\Gal}{\operatorname{Gal}}
\newcommand{\SL}{\operatorname{SL}}
\newcommand{\GL}{\operatorname{GL}}
\newcommand{\gSL}{\operatorname{\mathbf{SL}}}
\newcommand{\DM}{\operatorname{\mathbf{DM}}}
\newcommand{\gO}{\operatorname{\mathbf{O}}}
\newcommand{\gGL}{\operatorname{\mathbf{GL}}}
\newcommand{\gPGU}{\operatorname{\mathbf{PGU}}}
\newcommand{\gSpin}{\operatorname{\mathbf{Spin}}}
\newcommand{\End}{\operatorname{End}}
\newcommand{\Hom}{\operatorname{Hom}}
\newcommand{\Aut}{\operatorname{Aut}}
\newcommand{\Mor}{\operatorname{Mor}}
\newcommand{\Map}{\operatorname{Map}}
\newcommand{\Ext}{\operatorname{Ext}}
\newcommand{\Nrd}{\operatorname{Nrd}}
\newcommand{\ord}{\operatorname{ord}}
\newcommand{\ra}{\rightarrow}
\newcommand{\xra}{\xrightarrow}


\newcommand{\A}{\mathbb{A}}
\renewcommand{\P}{\mathbb{P}}
\newcommand{\p}{\mathbb{P}}
\newcommand{\Z}{\mathbb{Z}}
\newcommand{\z}{\mathbb{Z}}
\newcommand{\N}{\mathbb{N}}
\newcommand{\Q}{\mathbb{Q}}
\newcommand{\QZ}{\mathop{\mathbb{Q}/\mathbb{Z}}}
\newcommand{\gm}{\mathbb{G}_m}
\newcommand{\hh}{\mathbb{H}}

\newcommand{\cA}{\mathcal A}
\newcommand{\cB}{\mathcal B}
\newcommand{\cF}{\mathcal F}
\renewcommand{\cH}{\mathcal H}
\newcommand{\cJ}{\mathcal J}
\newcommand{\cM}{\mathcal M}
\newcommand{\cN}{\mathcal N}
\newcommand{\cO}{\mathcal O}
\renewcommand{\cR}{\mathcal R}
\newcommand{\cS}{\mathcal S}
\newcommand{\cX}{\mathcal X}
\newcommand{\cU}{\mathcal U}
\renewcommand{\cL}{\mathcal L}


\newcommand{\falg}{F\mbox{-}\mathfrak{alg}}
\newcommand{\fgroups}{F\mbox{-}\mathfrak{groups}}
\newcommand{\fields}{F\mbox{-}\mathfrak{fields}}
\newcommand{\groups}{\mathfrak{Groups}}
\newcommand{\abelian}{\mathfrak{Ab}}

\section
{$A$-triviality of standard norm varieties}
\label
{section A-triviality of standard norm varieties}

\subsection
{Retract rational varieties}

A variety $X$ over $F$ is called \emph{retract rational} if there
exist rational morphisms $\alpha:X\dashrightarrow \P^n$ and
$\beta:\P^n\dashrightarrow X$ for some $n$ such that the composition
$\beta\circ \alpha$ is defined and is equal to the identity of $X$.

The following proposition is due to D.~Saltman.

\begin{prop}
\label{retract}
Let $A$ be a central simple algebra of prime degree. Then the variety
of the algebraic group $\gSL_1(A)$ is retract rational.
\end{prop}

\begin{proof}
As $\deg(A)$ is prime, the group $SK_1(A_K)$ is trivial for every field extension $K/F$
\cite[\S23, Corollary 4]{MR696937}.
Taking $K=F(G)$, we can write the generic point $\xi\in \gSL_1(A)(K)$ as product
of commutators
\[
\xi=[f_1,f'_1]\cdot [f_2,f'_2]\cdot \ldots \cdot [f_n,f'_n]
\]
in $(A_K)\m$, where $f_i,g_i\in (A_K)\m$. The $2n$-tuple of functions $(f_1,f'_1,\dots, f_n,f'_n)$ can be viewed
as a rational morphism $\alpha$ from $G$ to the affine space $\A(A^{2n})$ of the direct sum of $2n$ copies of
the vector space of the algebra $A$. Define the rational morphism $\beta: \A(A^{2n})\dashrightarrow G$ by
\[
\beta(a_1,a'_1,\dots, a_n,a'_n)=[a_1,a'_1]\cdot\ldots\cdot [a_n,a'_n].
\]
By construction, the composition
$\beta\circ \alpha$ is defined and is equal to the identity of $G$.
\end{proof}

\subsection
{$A$-trivial varieties revisited}

Recall that a smooth complete variety $X$ over $F$ is called {\em $A$-trivial} if for every
field extension $K/F$ such that $X(K)\neq\emptyset$,
the degree homomorphism $\CH_0(X_K)\to \Lambda$ is an isomorphism.

\begin{example}
\label{rrA}
A retract rational smooth complete variety $X$ is $A$-trivial.
Indeed, it suffices to prove this for $\Lambda=\Z$.
Let $\alpha:X\dashrightarrow \P^n$ and
$\beta:\P^n\dashrightarrow X$ be rational morphisms such that the composition
$\beta\circ \alpha$ is defined and is equal to the identity of $X$, then the composition
(see Appendix \ref{Rational correspondences})
\[
\CH_0(X_K)\xra{\alpha_*} \CH_0(\P_K^n)\xra{\beta_*} \CH_0(X_K)
\]
is the identity for any field extension $K/F$.
As $\CH_0(\P_K^n)=\Z$, $\alpha_*$ is an isomorphism that is equal to the degree map.
\end{example}

\begin{prop}\label{a-trivial}
A smooth complete variety $X$ over $F$ is $A$-trivial if and only if for every
field extension $K/F$ and every two points $x,x'\in X(K)$, we have $[x]=[x']$ in $\CH_0(X_K)$.
\end{prop}
\begin{proof}
$\Rightarrow$: As $\deg[x]=1=\deg[x']$, we have $[x]=[x']$ in $\CH_0(X_K)$.

\noindent $\Leftarrow$: Let $K/F$ be a field extension such that $X(K)\neq\emptyset$. Let $x\in X_K$ be a rational point
and $y\in X_K$ a closed point of degree $m$. It suffices to show that $[y]=m[x]$ in $\CH_0(X_K)$. Let $L=K(y)$ and $y'$
a rational point of $X_L$ over $y$. By assumption $[y']=[x_L]$ in $\CH_0(X_L)$. Applying the push-forward homomorphism
$\CH_0(X_L)\to \CH_0(X_K)$, we get $[y]=m[x]$.
\end{proof}

\begin{proposition}\label{a-trivial2}
\noindent {\rm (1)} If $X$ is an $A$-trivial variety over $F$, then so is $X_K$ for any field extension $K/F$.

\noindent {\rm (2)} If $X$ and $X'$ are $A$-trivial varieties over $F$, then so is $X\times X'$.

\noindent {\rm (3)} Let $E/F$ be a separable field extension and $Y$ a variety over $E$.
If $Y$ is $A$-trivial, then so is the Weil transfer $R_{E/F}(Y)$.
\end{proposition}

\begin{proof}
(1) is trivial.

(2) Let $K/F$ be a field extension and $y_1, y_2\in (X\times X')(K)$. We have $y_1=(x_1,x'_1)$ and $y_2=(x_2,x'_2)$ for
$x_1,x_2\in X(K)$, $x_1',x_2'\in X'(K)$. As $X$ and $X'$ are $A$-trivial, we have
$[x_1]=[x_2]$ in $\CH_0(X_K)$, $[x'_1]=[x'_2]$ in $\CH_0(X'_K)$ and hence
\[
[y_1]=[x_1]\times [x'_1]=[x_2]\times [x'_2]=[y_2]
\]
in $\CH_0(X\times X')_K$.

(3) Let $K/F$ be a field extension and write $E\tens_F K\simeq E_1\times\dots\times E_s$, where $E_i$ are field extensions of $K$.
We have $R_{E/F}(Y)_K=R_{E\tens_F K/K}(Y_{E\tens_F K})=\prod R_{E_i/K}(Y_{E_i})$.Write two points
$y,y'\in R_{E/F}(Y)(K)=Y(E\tens_F K)=\prod Y(E_i)$ in the form $y=(y_1,\dots, y_s)$ and $y'=(y'_1,\dots, y'_s)$, where $y_i,y'_i\in Y(E_i)=R_{E_i/K}(Y_{E_i})(K)$.
By assumption, $[y_i]=[y'_i]$ in $\CH_0(Y_{E_i})$ for all $i$. Applying  the canonical maps
\[
\CH_0(Y_{E_i})\to \CH_0\(R_{E_i/K}(Y_{E_i})\)
\]
(see \cite{MR1809664}) we get $[y_i]=[y'_i]$ in $\CH_0(R_{E_i/K}(Y_{E_i}))$ for all $i$.
It follows that $[y]=[y']$ in $\CH_0(R_{E/F}(Y)_K)$.
\end{proof}

\subsection{Symmetric powers}

Let $p$ be an integer and let $Y$ be a quasi-projective variety over $F$.
Write $S^p(Y)$ for the symmetric $p$th power of $Y$, i.e.,
the factor variety of $Y^p$ by the natural action of the symmetric group $S_p$. A $K$-point of $S^pY$
is an effective $0$-cycle on $Y_K$ of degree $p$.

Write
$\widetilde{Y}^p$ for $Y^p$ with all the diagonals removed and write
$\widetilde S^p(Y)$ for the factor variety $\widetilde{Y}^p/S_p$.
If $Y$ is smooth, then so is $\widetilde S^p(Y)$. Every $F$-point $z$ of
$\widetilde S^p(Y)$ gives rise to an effective $0$-cycle $y_1+y_2+\dots + y_k$ on $Y$ of degree $p$ with distinct
closed points $y_1,y_2,\dots, y_k$ on $Y$. We will write $F\{z\}$ for the $F$-algebra
$F(y_1)\times F(y_2)\times   \dots\times F(y_k)$ of dimension $p$.

Consider the natural morphism $g:Y\times S^{p-1}(Y)\to S^p(Y)$. By \cite[\S2]{MR2220090}, the sheaf
$\cL=g_*(\cO_{Y\times S^{p-1}(Y)})|_{\widetilde S^p(Y)}$ is a locally free $\cO_{\widetilde S^p(Y)}$-algebra of
rank $p$.
The sheaf $\cL$ determines a rank $p$ vector bundle $J$ over
$\widetilde S^p(Y)$ such that the fiber of $J$ over a point $z$ in $\widetilde S^p(Y)$ is the vector space
of the algebra $F\{z\}$.

Consider the pull-back diagram
\[
\begin{CD}
\widetilde Y^p\coprod \dots\coprod \widetilde Y^p @>>> \widetilde Y^p  \\
@VVV  @VhVV     \\
g\inv(\widetilde S^p(Y))   @>>> \widetilde S^p(Y),      \\
\end{CD}
\]
where the left top corner is disjoint union of $p$ copies of $\widetilde Y^p$, the upper map is the
identity on each copy and the left vertical morphism on the $i$th copy takes a point $(y_1,\dots, y_p)$ to
$(y_i,y_1+\dots +\widehat y_i+\dots +y_p)$. It follows that the sheaf $h^*(\cL)$ of modules on $\widetilde Y^p$ is free.
Therefore, the vector bundle $h^*(J)$ on $\widetilde Y^p$ is trivial.

Let $E$ be an \'etale $F$-algebra of dimension $p$. Consider the natural morphism $f:R_{E/F}(Y_E)\to S^p(Y)$
(see \cite[pp. 267--268]{MR1611822}).
The map $f$ takes a point $y$ in $R_{E/F}(Y_E)(F)=Y(E)$ to the $0$-cycle $y_1+\dots+y_p$, where $y_i$ are the images of $y$ under
all the embeddings of $E$ into $F_\sep$.
If $E$ is split, $f$ is the natural morphism $Y^p\to S^p(Y)$. In general, $f$ can be obtained via the twist of this morphism by the
$S_p$-torsor corresponding to the \'etale algebra $E$. Moreover, the \'etale group scheme $G=\Aut(E/F)$ (the twisted form of the symmetric group $S_p$
by the same $S_p$-torsor) acts naturally on $R_{E/F}(Y_E)$ and $f$ identifies $S^p(Y)$ with the factor variety of $R_{E/F}(Y_E)$ by $G$.

Write $\widetilde R_{E/F}(Y_E)$ for the preimage $f\inv\(\widetilde S^p(Y)\)$ and $\tilde f$ for the
morphism $\widetilde R_{E/F}(Y_E)\to \widetilde S^p(Y)$.
\begin{lemma}\label{lift}
If $z$ is a rational point of $\widetilde S^p(Y)$ and $E=F\{z\}$, then $z$ lifts
to a rational point in $\widetilde R_{E/F}(Y_E)$.
\end{lemma}
\begin{proof}
We may assume that $E$ is a field. Over $F_\sep$, $z$ is the cycle $y_1+\dots + y_p$, where the distinct points $y\in Y(F_\sep)$
are permuted transitively by the Galois group of $F_\sep/F$. Choose a point $y\in Y(E)=\widetilde R_{E/F}(Y_E)$ and an embedding of $E$ into $F_\sep$
such that the map $Y(E)\to Y(F_\sep)$ takes $y$ to $y_1$. Then $\tilde f(y)=y_1+\dots +y_p=z$.
\end{proof}

Twisting the diagram above by the $S_p$-torsor $E$, we see that
the pull-back $\tilde f^*(J)$ is a trivial vector bundle over $\widetilde R_{E/F}(Y_E)$
with the fiber $E$, i.e., we have the following fiber product diagram:
\[
\begin{CD}
\widetilde R_{E/F}(Y_E)\times E   @>>> \widetilde R_{E/F}(Y_E)  \\
@VVV  @V{\tilde f}VV     \\
J   @>>> \widetilde S^p(Y).      \\
\end{CD}
\]

\subsection
{$A$-triviality of standard norm varieties}
\label{Subsection $A$-triviality of standard norm varieties}

We need $\ch F=0$ here.
We recall the construction of certain norm varieties given in \cite[\S2]{MR2220090}.

Let $p$ be a prime integer, $Y$ a variety over $F$ and $a\in F\m$.
Recall that we have a vector bundle $J$ of rank $p$ over $\widetilde S^p(Y)$.
Write $V(Y,a)$ for the hypersurface in $J$ defined in the fiber over every point $z$ of $\widetilde S^p(Y)$ by the equation $N=a$, where $N$ is the norm
map for the algebra $F\{z\}$.
If $Y$ is smooth geometrically irreducible variety,
then so is $V(Y,a)$ and $\dim V(Y,a)=p(\dim Y +1)-1$ by \cite[Lemma 2.1]{MR2220090}.

Let $E$ be an \'etale $F$-algebra of dimension $p$ and let $E^a$ be the hypersurface in the affine space $\A(E)$
given by the equation $N_{E/F}(x)=a$. We have the following fiber product diagram
\[
\begin{CD}
\widetilde R_{E/F}(Y_E)\times E^a   @>>> \widetilde R_{E/F}(Y_E)  \\
@VVV  @V{\tilde f}VV     \\
V(Y,a)   @>>> \widetilde S^p(Y).      \\
\end{CD}
\]
The \'etale group scheme $G=\Aut(E/F)$ acts naturally on the varieties in the top row
and the vertical morphisms are $G$-torsors.

Let $L$ be a cyclic \'etale $F$-algebra of degree $p$ and let $a_1,a_2,\dots, a_n$ (where $n\geq1$)
be a sequence of elements in $F\m$.
We define {\em a standard norm variety $W(L, a_1,a_2,\dots, a_n)$
for the sequence $(L,a_1,a_2,\dots, a_n)$} inductively as follows.
Let $W(L,a_1)$ be the Severi-Brauer variety for the cyclic algebra $(L/F,a_1)$ of degree $p$ and
for $n>1$ let
$W(L, a_1,a_2,\dots, a_n)$ be a smooth compactification of the smooth variety $V\(W(L,a_1,a_2,\dots, a_{n-1}), a_n\)$.
Note that $W(L,a_1,a_2,\dots, a_n)$ is a smooth projective geometrically irreducible variety over
$F$ of dimension $p^n-1$. Note that the birational class of a standard norm variety $W(L, a_1,a_2,\dots, a_n)$ is uniquely determined by the sequence
$(L,a_1,a_2,\dots, a_n)$.

Note that if $F$ contains a primitive $p$th root of unity $\xi_p$, we have $L=F(a_0^{1/p})$ for some $a_0\in F\m$ and
$W(L, a_1,a_2,\dots, a_n)=W(a_0, a_1,a_2,\dots, a_n)$ as defined in \cite[\S2]{MR2220090}.

For small $n$ one can identify the standard norm varieties varieties as follows
(see \cite{MR2714019} or \cite{Nguyen}):

\begin{example}\label{dva}
By definition, $W(L,a_1)$ is the Severi-Brauer variety for the cyclic algebra $A=(L/F,a_1)$.
If $W(L,a_1)_K$ has a point over a field extension $K/F$, it is isomorphic to the projective space
$\P_K^{p-1}$.
Hence it is rational over $K$ and therefore $W(L,a_1)$ is $A$-trivial.
\end{example}

\begin{example}\label{tri}
The variety $W(L,a_1,a_2)$ is birationally isomorphic to the $\gSL_1(A)$-torsor given by the equation $\Nrd=a_2$ in the cyclic algebra $A=(L/F,a_1)$,
where $\Nrd$ is the reduced norm map for $A$. If $W(L,a_1,a_2)$ has a rational point, it is birationally isomorphic
to the variety of the group  $\gSL_1(A)$ and hence is retract rational by Proposition \ref{retract}. It follows that $W(L,a_1,a_2)$
is $A$-trivial (see Example \ref{rrA}).
\end{example}

A sequence $(L,a_1,a_2,\dots,a_n)$ as above determines the {\em symbol}
\[
\{L,a_1,a_2,\dots,a_n\}:=(L)\cup (a_1)\cup\dots\cup (a_n)
\]
in $H^{n+1}(F,\mu_p^{\tens n})$, where $(L)\in H^1(F,\Z/p\Z)$ and $(a_i)\in H^1(F,\mu_p)$ are the
obvious classes.
By \cite{MR2220090}, the standard norm variety of the sequence $(L,a_1,a_2,\dots,a_n)$
is a norm variety of the symbol $\{L,a_1,a_2,\dots,a_n\}$.

\begin{theorem}
\label{chowmain}
Let $F$ be a field of characteristic $0$, $p$ a prime integer, $L/F$ a cyclic field extension of degree $p$ and $a_1,\dots, a_n\in F\m$.
If the symbol $\{L,a_1,a_2,\dots,a_n\}$ is nontrivial, then a standard norm variety
$X:=W(L,a_1,a_2,\dots, a_n)$
is $A$-trivial for the coefficient ring $\Lambda=\Z_{(p)}$.
Moreover, the degree map $\deg:\CH_0(X)\to \Z_{(p)}$ is injective.
\end{theorem}

\begin{proof}
In the proof we may assume that $n\geq 2$ (see Example \ref{dva}). We induct on $n$.

Write for simplicity
$$
Y=W(L,a_1,a_2,\dots, a_{n-1})\;\;\text{ and }\;\; U=V(Y,a_n).
$$
Thus, $X$ is a smooth
compactification of $U$.

\smallskip

{\em Claim:} Let $K/F$ be a field extension such that the symbol $\{L,a_1,\dots, a_{n-1}\}$ is nontrivial and $x_1,x_2\in U(K)$.
Then $[x_1]=[x_2]$ in $\CH_0(X)$ (with the coefficient ring $\Lambda=\Z_{(p)}$).

\smallskip

Replacing $F$ by $K$, we may assume that $K=F$. Moreover, we can also assume that
$F$ is $p$-special, in particular, $F$ contains $\xi_p$.
We shall write $\{a_0,a_1,\dots, a_{n-1}\}$ for $\{L,a_1,\dots, a_{n-1}\}$.

Let $E_i=F\{z_i\}$, $i=1,2$, where $z_i$ is the image of $x_i$ under the morphism $U\to \widetilde S^p(Y)$.
Each $E_i$ is a cyclic field extension of $F$ of degree $p$. By Lemma \ref{lift}, $Y(E_i)\neq\emptyset$,
and hence $E_i$ splits the symbol $\{a_0,a_1,\dots, a_{n-1}\}$ as $Y$ is a splitting variety for this
symbol.
By \cite[Theorem 5.6]{MR2220090}, there are $b_0, b_1\in F\m$ such that the symbol $\{b_0,b_1\}$
divides $\{a_0,a_1,\dots, a_{n-1}\}$ and both $E_i$ split $\{b_0,b_1\}$.
Let $Y'$ be the Severi-Brauer variety for the cyclic algebra $(F(b_0^{1/p})/F,b_1)$.
Let $U'=V(Y',a_n)$ and let $X'$ be a smooth compactification of $U'$, i.e. $X'$ is a standard norm variety for the $3$-sequence
$(b_0,b_1,a_n)$.

The function field $F(Y')$ splits the symbol $\{b_0,b_1\}$ and hence also splits $\{a_0,a_1,\dots, a_{n-1}\}$.
As $Y$ is a $p$-generic splitting variety
of the symbol $\{a_0,a_1,\dots, a_{n-1}\}$,
there is a finite field extension $M/F(Y')$ of degree prime to $p$ such that $Y$ has a point over $M$.
Choose a
smooth projective model $Y''$ of $M$ over $F$.
There are
two rational maps
\[
Y'\dashleftarrow Y'' \dashrightarrow Y,
\]
where the left map is dominant of degree prime to $p$.
The right one is not constant as $Y$ has no $F$-point. Therefore, the image of the induced
rational map $S^p(Y'')\dashrightarrow S^p(Y)$ intersects $\widetilde S^p(Y)$ nontrivially and hence we have two rational maps
$\widetilde R_{E_i/F}(Y''_{E_i})  \dashrightarrow \widetilde R_{E_i/F}(Y_{E_i})$ (for $i=1,2$).
Thus, there are two commutative diagrams:
\[
\xymatrix{\widetilde R_{E_i/F}(Y'_{E_i})\times E_i^{a_n} \ar[d] & \widetilde R_{E_i/F}(Y''_{E_i})\times E_i^{a_n}  \ar[d]\ar@{-->}[l]\ar@{-->}[r] & \widetilde R_{E_i/F}(Y_{E_i})\times E_i^{a_n} \ar[d] \\
U'  & U''  \ar@{-->}[l]\ar@{-->}[r] &  U, }
\]
where $U''=V(Y'',a_n)$.

The variety $E_i^{a_n}$ is a torsor under the norm one torus $T_i$ for the cyclic extension $E_i/F$ of degree $p$.
As $T_i\simeq R_{E_i/F}({\gm}_{,E_i})/\gm$ and $E_i^{a_n}$ has a rational point,
we can embed the variety $E_i^{a_n}\simeq T_i$ as an open subset into the projective space $\P(E_i)$.
Thus, we have the following commutative diagram of
rational maps of smooth projective varieties:
\[
\xymatrix{R_{E_i/F}(Y'_{E_i})\times \P(E_i) \ar@{-->}[d]_{\beta'_i} & R_{E_i/F}(Y''_{E_i})\times \P(E_i)  \ar@{-->}[d]_{\beta''_i}\ar@{-->}[l]_{\varepsilon'_i}\ar@{-->}[r]^{\varepsilon_i} & R_{E_i/F}(Y_{E_i})\times \P(E_i) \ar@{-->}[d]_{\beta_i} \\
X'  & X''  \ar@{-->}[l]_{\nu'}\ar@{-->}[r]^{\nu} &  X, }
\]
where $X''$ is a smooth compactification of $U''$.

By Proposition \ref{trans}, we have the equality $(\nu')^t\circ\beta'_i=\beta''_i\circ (\varepsilon'_i)^t$ of rational correspondences.
Hence we get a commutative diagram of rational correspondences between smooth projective varieties:
\[
\xymatrix{R_{E_i/F}(Y'_{E_i})\times \P(E_i) \ar@{~>}[d]_{\beta'_i}\ar@{~>}[r]^{\alpha_i}  & R_{E_i/F}(Y_{E_i})\times \P(E_i) \ar@{~>}[d]_{\beta_i} \\
X' \ar@{~>}[r]^{\delta}  & X, }
\]
where $\alpha_i=\varepsilon_i\circ (\varepsilon'_i)^t$ and $\delta=\nu\circ (\nu')^t$.

Note that the correspondences $\beta$ and $\beta_i$ are of multiplicity 1
and the correspondences $\alpha_i$ and $\delta$ are of multiplicity prime to $p$.

Choose rational points $w_i$ in $R_{E_i/F}(Y'_{E_i})\times \P(E_i)$. As $X'$ is $A$-trivial by Example \ref{tri} and the $0$-cycles ${\beta'_i}_*([w_i])$ both
have degree $1$, we have ${\beta'_1}_*([w_1])={\beta'_2}_*([w_2])$ in $\CH_0(X')$. It follows that
\begin{equation}\label{comp}
{\beta_1}_*({\alpha_1}_*([w_1]))={\delta}_*({\beta'_1}_*([w_1]))={\delta}_*({\beta'_2}_*([w_2]))={\beta_2}_*({\alpha_2}_*([w_2])).
\end{equation}
By Lemma \ref{lift}, there are rational points $y_i$ in $\widetilde R_{E_i/F}(Y_{E_i})\times E^{a_n}$ over $x_i$.
It follows from Lemma \ref{point} that ${\beta_i}_*([y_i])=[x_i]$ in $\CH_0(X)$.

Since by the induction hypothesis $Y$ is $A$-trivial,
it follows from Proposition \ref{a-trivial2} that the varieties
$R_{E_i/F}(Y_{E_i})\times \P(E_i)$ are $A$-trivial, hence
\begin{equation}\label{ischo}
{\alpha_i}_*([w_i])=m[y_i],
\end{equation}
where $m=\mult(\delta)=\mult(\alpha_i)=\deg {\alpha_i}_*([w_i])$. It follows from (\ref{comp}) and (\ref{ischo}) that
\[
m[x_1]=m{\beta_1}_*([y_1])={\beta_1}_*({\alpha_1}_*([w_1]))={\beta_2}_*({\alpha_2}_*([w_2]))=m{\beta_2}_*([y_2])=m[x_2]
\]
in $\CH_0(X)$. Since $m$ is prime to $p$, the claim is proved.

Now we finish the proof of the first part of
the theorem.
Let $K/F$ be a field extension and $x_1,x_2\in X(K)$. By Proposition \ref{a-trivial},
it suffices to show that $[x_1]=[x_2]$ in $\CH_0(X_K)$.

Consider the field $L=F(X\times X)$ and the two "generic" points $\xi_1$ and $\xi_2$ in $X(L)$. Note that
$\xi_i\in U(L)$ and
the field $F(X\times X)$ does not split the symbol $\{a_1,\dots,
a_{n-1}\}$: otherwise the variety $Y$ would be equivalent (in the sense of
Section \ref{Abstract Rost motives}) to the varieties $X\sim X\times X$
\newcommand{\cd}{\mathop{\mathrm{cd}}\nolimits}
and we would have $\cd_p Y=\cd_p X$ by Lemma \ref{cdp=cdp}
contradicting $\cd_p Y=p^{n-1}-1<p^n-1=\cd_p X$ (see Corollary \ref{aRinc}).
By the claim, $[\xi_1]=[\xi_2]$ in $\CH_0(X_L)$.
Specializing $\xi_i$ to $x_i$ (see \cite[\S20.3]{Fulton} or \cite{MR1418952}), we get the result.

We prove now that  the degree map $\deg:\CH_0(X)\to \Z_{(p)}$ is injective.
We may assume that $F$ is a $p$-special field.
We claim that if $x$ and $x'$ are two closed points of $X$ of degree $p$,
then $[x]=[x']$ in $\CH_0(X)$.
To prove the claim,
as in the first part of the proof, we find a Severi-Brauer variety $Y$ over $F$,
that is split by the two field extensions
$L:=F(x)$ and $L':=F(x')$ of $F$, and a  correspondence $\delta: Y\rightsquigarrow X$ of degree $m$ prime to $p$.
Let $y$ and $y'$
be closed points of $Y$ with residue fields isomorphic to $L$ and $L'$ respectively.

Consider the
correspondence $\delta_L: Y_L\rightsquigarrow X_L$ and rational points
$y_1\in Y_L$, $x_1\in X_L$ over $y$ and $x$ respectively.
As $\delta_*([y_1])$ is a $0$-cycle on $X_L$ of degree $m$ and the variety $X$ is $A$-trivial by the first part
of the proof, we have $(\delta_L)_*([y_1])=m[x_1]$.
Taking the norms for the extension $L/F$, we get  $\delta_*([y])=m[x]$
in $\CH_0(X)$.
Similarly, $\delta_*([y'])=m[x']$.
As $[y]=[y']$ in $\CH_0(Y)$ by \cite{Panin-thesis} or \cite{MR675529}, we have
$[x]=[x']$ in $\CH_0(X)$.
The claim is proved.

Take any closed point $x\in X$ and set $\deg(x)=p^k$ for some $k>0$.
As $\deg\CH_0(X)=p\Z_{(p)}$ and $F$ is $p$-special,
there is a closed point $x'\in X$
of degree $p$.
It suffices to show that $[x]=p^{k-1}[x']$ in $\CH_0(X)$.
Choose a field $K$ with $F\subset K\subset F(x)$ and
$[F(x):K]=p$.
Let $x_1$ be a closed point in $X_K$ over $x$ with $\deg(x_1)=p$.
We shall show that $[x_1]=[x']_K$ in $\CH_0(X_K)$
and then taking the norms for the extension $K/F$ we get the desired equality.
As both cycles $[x_1]$ and $[x']_K$ have degree $p$,
they are equal in the case $X(K)\neq\emptyset$ as $X$ is $A$-trivial.
In the case $X(K)=\emptyset$ the cycles $[x_1]$ and $[x']_K$ are equal by the claim
as $x_1$ and $x'_K$ are points in $X_K$ of degree $p$.
\end{proof}

}

\renewcommand{\thesubsection}{\thesection-\Roman{subsection}}

{


\renewcommand{\(}{\bigl(}
\renewcommand{\)}{\bigr)}


\newcommand{\m}{^{\times}}
\newcommand{\ff}{F^{\times}}
\newcommand{\fs}{F^{\times 2}}
\newcommand{\llg}{\longrightarrow}
\newcommand{\tens}{\otimes}
\newcommand{\inv}{^{-1}}


\newcommand{\Dfn}{\stackrel{\mathrm{def}}{=}}
\newcommand{\iso}{\stackrel{\sim}{\to}}


\newcommand{\mult}{\mathrm{mult}}
\newcommand{\Gam}{\mathrm{\Gamma}}
\newcommand{\cchar}{\mathrm{char}}
\newcommand{\sep}{\mathrm{sep}}
\newcommand{\tors}{\mathrm{tors}}
\newcommand{\id}{\mathrm{id}}
\newcommand{\diag}{\mathrm{diag}}
\newcommand{\disc}{\mathrm{disc}}
\newcommand{\op}{^{\mathrm{op}}}
\newcommand{\ra}{\rightarrow}
\newcommand{\xra}{\xrightarrow}


\newcommand{\CH}{\operatorname{CH}}
\newcommand{\F}{\operatorname{F}}
\renewcommand{\Im}{\operatorname{Im}}
\newcommand{\Ad}{\operatorname{Ad}}
\newcommand{\NN}{\operatorname{N}}
\newcommand{\Ker}{\operatorname{Ker}}
\newcommand{\Pic}{\operatorname{Pic}}
\newcommand{\Tor}{\operatorname{Tor}}
\newcommand{\Lie}{\operatorname{Lie}}
\newcommand{\ind}{\operatorname{ind}}
\newcommand{\rank}{\operatorname{rank}}
\newcommand{\ch}{\operatorname{char}}
\newcommand{\Inv}{\operatorname{Inv}}
\newcommand{\Gras}{\mathop{\mathrm{Gr}}}
\newcommand{\res}{\operatorname{res}}
\newcommand{\Br}{\operatorname{Br}}
\newcommand{\Spec}{\operatorname{Spec}}
\newcommand{\SK}{\operatorname{SK}}
\newcommand{\SB}{\operatorname{SB}}
\newcommand{\Gal}{\operatorname{Gal}}
\newcommand{\SL}{\operatorname{SL}}
\newcommand{\Spin}{\operatorname{Spin}}
\newcommand{\GL}{\operatorname{GL}}
\newcommand{\SU}{\operatorname{SU}}
\newcommand{\gSL}{\operatorname{\mathbf{SL}}}
\newcommand{\gSp}{\operatorname{\mathbf{Sp}}}
\newcommand{\gSU}{\operatorname{\mathbf{SU}}}
\newcommand{\gO}{\operatorname{\mathbf{O}}}
\newcommand{\gGL}{\operatorname{\mathbf{GL}}}
\newcommand{\gPGU}{\operatorname{\mathbf{PGU}}}
\newcommand{\gSpin}{\operatorname{\mathbf{Spin}}}
\newcommand{\End}{\operatorname{End}}
\newcommand{\Hom}{\operatorname{Hom}}
\newcommand{\Mor}{\operatorname{Mor}}
\newcommand{\Map}{\operatorname{Map}}
\newcommand{\Ext}{\operatorname{Ext}}
\newcommand{\RatCor}{\operatorname{RatCor}}
\newcommand{\Cor}{\operatorname{Cor}}
\newcommand{\Var}{\operatorname{Var}}
\newcommand{\Ab}{\operatorname{Ab}}


\newcommand{\A}{\mathbb{A}}
\renewcommand{\P}{\mathbb{P}}
\newcommand{\Z}{\mathbb{Z}}
\newcommand{\N}{\mathbb{N}}
\newcommand{\OO}{\mathbb{O}}
\newcommand{\Q}{\mathbb{Q}}
\newcommand{\QZ}{\mathop{\mathbb{Q}/\mathbb{Z}}}
\newcommand{\gm}{\mathbb{G}_m}
\newcommand{\hh}{\mathbb{H}}

\newcommand{\cA}{\mathcal A}
\newcommand{\cU}{\mathcal U}
\newcommand{\cI}{\mathcal I}
\newcommand{\cJ}{\mathcal J}
\newcommand{\cK}{\mathcal K}
\newcommand{\cO}{\mathcal O}
\newcommand{\cE}{\mathcal E}
\newcommand{\cF}{\mathcal F}
\renewcommand{\cH}{\mathcal H}
\newcommand{\cM}{\mathcal M}
\newcommand{\cP}{\mathcal P}


\newcommand{\falg}{F\mbox{-}\mathfrak{alg}}
\newcommand{\fgroups}{F\mbox{-}\mathfrak{groups}}
\newcommand{\fields}{F\mbox{-}\mathfrak{fields}}
\newcommand{\groups}{\mathfrak{Groups}}
\newcommand{\abelian}{\mathfrak{Ab}}

\newcommand{\Label}{\label}

\appendix
\addtocounter{section}{17}
\renewcommand{\thesection}{RC}

\section
{Rational correspondences}
\label{Rational correspondences}

We review the construction of the category of rational correspondences due to M.~Rost (unpublished) and
B.~Kahn/R.~Sujatha in \cite{KS}. We give Rost's approach using cycle modules.

\subsection{Integral correspondences}

Let $M$ be a cycle module over a field $F$ and let $X$ be an
algebraic variety over $F$. The groups
$$
C_p(X;M)=\coprod_{x\in X_{(p)}}M(x),
$$
where $M(x)=M\(F(x)\)$, form a {\it cycle complex} $C(X;M)$
\cite[3.2]{MR1418952}. Denote by $A_p(X;M)$ the homology groups of
$C(X;M)$. If $X$ is equidimensional of dimension $d_X$ we set
$$
A^q(X;M)=A_{d_X-q}(X;M).
$$

\begin{example}
If $M=K$ is given by Milnor's $K$-theory of fields, $A_p(X;K)$ are
the Milnor $K$-cohomology groups of $X$, in particular, the Chow groups
of $X$ (see \cite[Chapter IX]{EKM}).
\end{example}

Let $X,Y$ and $Z$ be algebraic varieties over $F$ with $X$
irreducible smooth and complete. We define a pairing ({\it
$\cup$-product})
\[
\CH_r(X\times Y)\tens A_p(Z\times X;M)\xra{\cup}
A_{r+p-d_X}(Z\times Y;M),\ \ \ (v,a)\mapsto v\cup a
\]
as the composition
\[
\CH_r(X\times Y)\tens A_p(Z\times X;M)\xra{\times}
A_{r+p}(X\times X\times Z\times Y;M)
\]
\[
\xra{(\Delta\times\id_{Y\times Z})^*} A_{r+p-d_X}(X\times Z\times
Y;M)\xra{q_*} A_{r+p-d_X}(Z\times Y;M),
\]
where $\times$ is the external product, $\Delta:X\to X\times X$ is
the diagonal embedding and $q:X\times Z\times Y\to Z\times Y$ is
the projection. (See \cite{MR1418952} for the definitions. We need $X$
smooth to get $\Delta$ a regular embedding and $X$ complete to
have $q$ proper.)
\begin{example}
Let $f:X\to Y$ and $g:Y\to X$ be morphisms with the graphs
$\Gamma_f\subset X\times Y$ and $\Gamma_g\subset Y\times X$. Then
$[\Gamma_f]\cup=( \id_{Z}\times f)_*$ and
$[\Gamma_g]^t\cup=(\id_{Z}\times g)^*$ (here $t$ is the
transposition involution).
\end{example}

In particular, we have the product
\[
\CH_r(X\times Y)\tens \CH_p(Z\times X)\xra{\cup}
\CH_{r+p-d_X}(Z\times Y).
\]
It is taken as the composition law for the category of {\it
integral correspondences} $\Cor(F)$ (see \cite{MR0258836} and \cite[Chapter XII]{EKM}) with the
objects smooth complete varieties over $F$ and morphisms
\[
\Mor_{\Cor(F)}(X,Y)=\coprod_{i}\CH_{d_i}(X_i\times Y),
\]
where $X_i$ are irreducible (connected) components of $X$ with
$d_i=\dim X_i$.

Denote by $\Var(F)$ the category of smooth complete varieties over
$F$ and morphisms of varieties. There is a natural functor
\[
\Var(F)\to \Cor(F),\ \ \ X\mapsto X,\ \ f\mapsto [\Gamma_f].
\]

The functors $\Var(F)\to \Ab$, $X\mapsto A_p(X;M)$ and $X\mapsto
A^p(X;M)$ factor through a covariant functor $\Cor(F)\to \Ab$,
\begin{equation}\label{dim}
X\mapsto A_p(X;M), \ \ \ a\mapsto a\cup
\end{equation}
and a contravariant functor
\begin{equation}\label{codim}
X\mapsto A^p(X;M), \ \ \ a\mapsto a^t\cup.
\end{equation}

\subsection{The cycle module $A_0[Y,M]$}

For a cycle module $M$ over $F$ and an algebraic variety $Y$ over
$F$ define the cycle module $A_0[Y,M]$ over $F$ by
\[
A_0[Y,M](L)=A_0(Y_L,M)
\]
(see \cite[\S7]{MR1418952}). There is a canonical map of complexes
\[
\theta_{Y,M}: C(X\times Y;M)\to C(X;A_0[Y,M]),
\]
that takes an element in $M(z)$ for $z\in (X\times Y)_{(p)}$ to
zero if dimension of the projection $x$ of $z$ in $X$ is less than
$p$ and identically to itself otherwise. In the latter case we
consider $z$ as a point of dimension $0$ in $Y_x:=Y_{F(x)}$ under the
inclusion $Y_x\subset X\times Y$. Thus, $\theta_{Y,M}$ ``ignores"
points in $X\times Y$ that lose dimension being projected to $X$.

We study various compatibility properties of $\theta$.
\subsubsection{Cross products}

Let $ N\times M\to P$ be a bilinear pairing of cycle modules over
$F$. For a variety $Y$ over $F$ we can define a pairing
\[
A_0[Y,N] \times M \to A_0[Y,P]
\]
in an obvious way.

\begin{lemma}\label{cross}
Let $M$ be a cycle module, $X,Y$ and $Z$ varieties over $F$. Then
the following diagram is commutative:
\[
\begin{CD}
C(X\times Y;N)\tens C(Z;M) @>{\times}>> C(X\times Y\times Z;P)  \\
@V{{\theta_{Y,N}}\tens\id}VV  @VV{\theta_{Y,P}}V     \\
C(X;A_0[Y,N])\tens C(Z;M)    @>{\times}>> C(X\times Z;A_0[Y,P]).     \\
\end{CD}
\]
\end{lemma}

\begin{proof}
Let $z\in Z_{(k)}$ and $\mu\in C(Z;M)$. Consider the following
commutative diagram
\small{
\[
\begin{CD}
C(X\times Y;N) @>{{\pi_z'}^*}>> C\((X\times Y)_z;N\)   @>{m'_\mu}>>  C\((X\times Y)_z;P\) @>{{i_z'}_*}>> C(X\times Y\times Z;P) \\
@V{\theta_{Y,N}}VV  @V{\theta_{Y,N}}VV  @ V{\theta_{Y,P}}VV  @V{\theta_{Y,P}}VV  \\
C(X;A_0[Y,N])    @>{\pi_z^*}>> C(X_z;A_0[Y,N])  @>{m_\mu}>>  C(X_z;A_0[Y,P])  @>{{i_z}_*}>>  C(X\times Z;A_0[Y,P]), \\
\end{CD}
\]
}
where $\pi_z:X_z\to X$ and $\pi'_z:(X\times Y)_z\to X\times Y$ are the
natural projections, $m_\mu$ and $m'_\mu$ are the multiplications
by $\mu$, $i_z:X_z\to X\times Z$ and $i_z':(X\times Y)_z\to
X\times Y\times Z$ are the inclusions. By the definition of the cross
product, the compositions in the two rows of the diagram are the
multiplications by $\mu$.
\end{proof}

\subsubsection{Pull-back maps}

Let $f:Z\hookrightarrow X$ be a regular closed embedding,
$N_{X/Z}$ the normal bundle over $Z$. Choose a coordination $\tau$
of $N_{X/Z}$ \cite[\S9]{MR1418952}. For a variety $Y$, the closed
embedding
\[
f'=f\times\id_Y:Z\times Y\hookrightarrow X\times Y
\]
is also regular and the normal bundle $N_{X\times Y/Z\times Y}$ is
isomorphic to $N_{X/Z}\times Y$. We choose the induced
coordination $\tau'=\tau\times Y$ of $N_{X\times Y/Z\times Y}$.

\begin{lemma}\label{pull-back}
For every cycle module $M$ the following diagram commutes:
\[
\begin{CD}
C(X\times Y;M)  @>{I(f')}>> C(Z\times Y;M)  \\
@V{\theta_{Y,M}}VV  @VV{\theta_{Y,M}}V     \\
C(X;A_0[Y,M])   @>{I(f)}>> C(Z;A_0[Y,M]),     \\
\end{CD}
\]
where the pull-back maps $I(f)$ and $I(f')$ are chosen with
respect to the coordinations $\tau$ and $\tau'$ respectively (see \cite[\S12]{MR1418952}).
\end{lemma}

\begin{proof}
Let $q:X\times\gm\to X$ and $q':X\times Y\times\gm\to X\times Y$
be the natural projections. The following diagram is clearly
commutative:
\[
\begin{CD}
C(X\times Y;M)  @>{{q'}^*}>>  C(X\times Y\times\gm;M)   \\
@V{\theta_{Y,M}}VV  @VV{\theta_{Y,M}}V     \\
C(X;A_0[Y,M])   @>{q^*}>> C(X\times\gm;A_0[Y,M])     \\
\end{CD}
\]
(here $q^*$ and ${q'}^*$ are the flat pull-back maps
\cite[3.5]{MR1418952}).

Let $t$ be the coordinate function on $\gm$. The map
$\theta_{Y,M}$ clearly commutes with the multiplication by $t$,
i.e. the following diagram is commutative:
\[
\begin{CD}
C(X\times Y\times\gm;M)  @>{(t)}>>  C(X\times Y\times\gm;M)   \\
@V{\theta_{Y,M}}VV  @VV{\theta_{Y,M}}V     \\
C(X\times\gm;A_0[Y,M])   @>{(t)}>> C(X\times\gm;A_0[Y,M]).     \\
\end{CD}
\]

Let $D=D(X,Z)$ be the deformation space of the embedding $f$
\cite[\S10]{MR1418952}. There is a closed embedding
$i:N_{X/Z}\hookrightarrow D$ with the open complement
$j:X\times\gm\hookrightarrow D$. Then $D'=D\times Y$ is the
deformation space $D(X\times Y,Z\times Y)$ with the closed
embedding
\[
i'=i\times \id_Y:N_{X\times Y/Z\times Y}\hookrightarrow
D'
\]
and the open complement $j'=j\times \id_Y:X\times Y\times\gm\hookrightarrow D'$.

The commutative diagram of complexes with the exact rows
\small{
\[
\begin{CD}
0   @>>>  C(N_{X/Z}\times Y;M)  @>{i'_*}>> C(D';M)   @>{{j'}^*}>> C(X\times Y\times\gm;M)   @>>>  0 \\
@. @V{\theta_{Y,M}}VV  @VV{\theta_{Y,M}}V   @V{\theta_{Y,M}}VV  \\
0  @>>>  C(N_{X/Z};A_0[Y,M])   @>{i_*}>> C(D;A_0[Y,M])   @>{j^*}>>   C(X\times\gm;A_0[Y,M])  @>>>  0  \\
\end{CD}
\]
}
induces the commutative diagram
\[
\begin{CD}
C(X\times Y\times\gm;M)  @>{\partial}>>  C(N_{X/Z}\times Y;M)   \\
@V{\theta_{Y,M}}VV  @VV{\theta_{Y,M}}V     \\
C(X\times\gm;A_0[Y,M])   @>{\partial}>> C(N_{X/Z};A_0[Y,M]).     \\
\end{CD}
\]

The coordinations $\tau$ and $\tau'$ induce the commutative
diagram \cite[\S9]{MR1418952}
\[
\begin{CD}
C(N_{X/Z}\times Y;M)   @>{r(\tau')}>>  C(Z\times Y;M)   \\
@V{\theta_{Y,M}}VV  @VV{\theta_{Y,M}}V     \\
C(N_{X/Z};A_0[Y,M])   @>{r(\tau)}>> C(Z;A_0[Y,M]).     \\
\end{CD}
\]

By the definition of the pull-back map, the diagram in question is the
composition of the four commutative square diagrams considered in the
proof.
\end{proof}

\subsubsection{Push-forward maps}

Let $f:X\to Z$ be a morphism of varieties over $F$. For a variety
$Y$ set
\[
f'=f\times \id_Y:X\times Y\to Z\times Y.
\]

\begin{lemma}\label{push-forward}
The following diagram is commutative:
\[
\begin{CD}
C(X\times Y;M)   @>{f'_*}>>  C(Z\times Y;M)   \\
@V{\theta_{Y,M}}VV  @VV{\theta_{Y,M}}V     \\
C(X;A_0[Y,M])  @>{f_*}>> C(Z;A_0[Y,M]).     \\
\end{CD}
\]
\end{lemma}

\begin{proof}
Let $u\in (X\times Y)_{(p)}$, $a\in M(u)$. Set $v=f'(u)\in Z\times
Y$. If $\dim(v)<p$ then $(f'_*)_u(a)=0$. In this case the dimension of
the projection $y$ of $u$ in $Y$ is less than $p$ and hence
$\theta_u(a)=0$.

Assume that $\dim(v)=p$. Then $F(u)/F(v)$ is a finite field
extension and
\[
b=(f'_*)_u(a)=c_{F(u)/F(v)}(a)\in M(v),
\]
where $c_{F(u)/F(v)}$ is the norm map.
If $\dim(y)<p$, then $\theta_u(a)=0$ and $\theta_v(b)=0$.

Assume that $\dim(y)=p$. Then
\[
(\theta\circ f'_*)_u(a)=c_{F(u)/F(v)}(a)=b
\]
considered as an element of $A_0[Y;M](z)=A_0(Y_{z};M)$, where $z$
is the image of $v$ in $Z$. On the other hand,
\[
(f_*\circ \theta)_u(a)=\varphi_*(a),
\]
where $\varphi:Y_{x}\to Y_{z}$ is the natural morphism ($x$ is the
image of $u$ in $X$) and $a$ is considered as an element of
$A_0[Y;M](x)$. It remains to notice that
\[
\varphi_*(a)=c_{F(u)/F(v)}(a)=b.\qedhere
\]
\end{proof}

\subsection{The category of rational correspondences}\label{crc}

Let $X$ and $Y$ be varieties over $F$ and let $M$ be a cycle
module over $F$. By Lemma \ref{cross}, for the pairing $K\times
M\to M$ and $Z=X$ we have the commutative diagram
\[
\begin{CD}
\CH_{d_X}(X\times Y)\tens A_0(X;M)   @>{\times}>>  A_{d_X}(X\times X\times Y;M)   \\
@V{{\theta_{Y,M}}\tens\id}VV  @VV{\theta_{Y,M}}V     \\
A_{d_X}(X;A_0[Y,K_0])\tens A_0(X;M)  @>{\times}>> A_{d_X}(X\times X;A_0[Y,M]).     \\
\end{CD}
\]

Assume that $X$ is smooth. Let $\Delta:X\to X\times X$ be the
diagonal embedding and $\Delta'=\Delta\times\id_Y$. By Lemma
\ref{pull-back}, the following diagram is commutative:
\[
\begin{CD}
A_{d_X}(X\times X\times Y;M)   @>{{\Delta'}^*}>>  A_0(X\times Y;M)   \\
@V{\theta_{Y,M}}VV  @VV{\theta_{Y,M}}V     \\
A_{d_X}(X\times X;A_0[Y,M])   @>{\Delta^*}>> A_0(X;A_0[Y,M]).     \\
\end{CD}
\]

Finally, assume that $X$ is complete. Let $f:X\to \Spec F$ be the
structure morphism and $f'=f\times\id_Y$. Lemma \ref{push-forward}
gives the following commutative diagram:
\[
\begin{CD}
A_0(X\times Y;M)   @>{f'_*}>>   A_0(Y;M)  \\
@V{\theta_{Y,M}}VV  @|     \\
A_0(X;A_0[Y,M])    @>{f_*}>> A_0(\Spec F;A_0[Y,M]).     \\
\end{CD}
\]

\begin{prop}\label{factor}
Let $X$ and $Y$ be varieties over $F$, let $X$ be irreducible
smooth and proper and let $M$ be a cycle module over $F$. Then the
pairing
\[
\CH_{d_X}(X\times Y)\tens A_0(X;M) \xra{\cup}  A_0(Y;M)
\]
is trivial on all cycles in $\CH_{d_X}(X\times Y)$ that are not
dominant over $X$. In other words, the $\cup$-product factors
through a natural pairing
\[
r:\CH_0(Y_{F(X)})\tens A_0(X;M) \xra{\cup}  A_0(Y;M).
\]
\end{prop}

\begin{proof}
Composing all three diagrams in Section \ref{crc} and taking into account that
\[
A_{d_X}(X;A_0[Y,K_0])=\CH_0(Y_{F(X)}),
\]
we get the commutative diagram
\[
\begin{CD}
\CH_{d_X}(X\times Y)\tens A_0(X;M)   @>{\cup}>>   A_0(Y;M)  \\
@V{r\times\id}VV  @|     \\
\CH_0(Y_{F(X)})\tens A_0(X;M)    @>{\cup}>> A_0(Y;M),    \\
\end{CD}
\]
whence the statement.
\end{proof}

In the conditions of Proposition \ref{factor}, for an irreducible
variety $Z$ over $F$ the diagram
\[
\begin{CD}
\CH_{d_X}(X\times Y)\tens A_0(Z\times X;M)   @>{\cup}>>  A_0(Z\times Y;M)  \\
@VVV  @VVV     \\
 \CH_{d_X}(X\times Y)\tens A_0(X_{F(Z)};M)    @>{\cup}>> A_0(Y_{F(Z)};M)     \\
 @V{r\times\id}VV  @|     \\
 \CH_0(Y_{F(X)})\tens A_0(X_{F(Z)};M)  @>{\cup}>>  A_0(Y_{F(Z)};M)
\end{CD}
\]
is commutative.

In particular, we have a well defined pairing
\[
\CH_0(Y_{F(X)})\tens \CH_0(X_{F(Z)})    \xra{\cup}
\CH_0(Y_{F(Z)})
\]
that can be taken for the composition law in the category of
{\it rational correspondences} $\RatCor(V)$ with the objects smooth
complete varieties over $F$ and the morphisms
\[
\Mor_{\RatCor(F)}(X,Y)=\coprod_{i}\CH_0(Y_{F(X_i)}),
\]
where $X_i$ are all irreducible (connected) components of $X$.

There is an obvious functor $\kappa:\Cor(F)\to \RatCor(F)$.

\begin{thm}\label{Rmain}
For a cycle module $M$ there are well defined:

\noindent {\rm (1)} The covariant functor
\[
\RatCor(F)\to \Ab,\ \ \ X\mapsto A_0(X;M),\ a\mapsto a\cup,
\]
i.e. the functor $(\ref{dim})$ factors through $\kappa$ if $p=0$.

\noindent {\rm (2)} The contravariant functor
\[
\RatCor(F)\to \Ab,\ \ \ X\mapsto A^0(X;M),\ a\mapsto a^t\cup,
\]
i.e. the functor $(\ref{codim})$ factors through $\kappa$ if
$p=0$.
\end{thm}
\begin{proof}
The first statement follows from Proposition \ref{factor}. To
prove the second part consider an irreducible variety $Y$ and an
open subset $j:U\hookrightarrow Y$. For a smooth complete $X$ in
the commutative diagram
\[
\begin{CD}
\CH_{d_Y}(X\times Y)\tens A^0(X;M)   @>{\cup}>>   A^0(Y;M)  \\
@V{(j\tens\id)^*}VV  @VV{j^*}V     \\
\CH_{d_Y}(X\times U)\tens A^0(X;M)    @>{\cup}>>  A^0(U;M)   \\
\end{CD}
\]
the right vertical homomorphism is injective. Hence the pairing in
the top row of the diagram is trivial on the cycles in
$\CH_{d_Y}(X\times Y)$ that are not dominant over $Y$. Thus we
have a well defined pairing
\[
\CH_{d_Y}(X_{F(Y)})\tens A^0(X;M)   \xra{\cup}   A^0(Y;M)
\]
that defines a contravariant functor
\[
X\mapsto A^0(X;M), \ \ \ a\mapsto a^t\cup.\qedhere
\]
\end{proof}

By Theorem \ref{Rmain}, for a cycle module $M$ and a rational correspondence $\alpha:X \rightsquigarrow Y$, we have the two natural
homomorphisms
\[
\alpha_*:=\alpha\cup: A_0(X;M)\to A_0(Y;M)
\]
and
\[
\alpha^*:=\alpha^t\cup: A^0(Y;M)\to A^0(X;M).
\]

Let $f:X\dashrightarrow Y$ be a rational morphism of irreducible varieties. It
defines a rational point of $Y_{F(X)}$ over $F(X)$ and hence a
morphism in $\Mor_{\RatCor}(X,Y)$ that we still denote by $f:X \rightsquigarrow Y$. In fact, the rational
correspondence $f$ is the image of the class of the graph of $f$ under the natural homomorphism
$\CH_{d_X}(X\times Y) \to \CH_0(Y_{F(X)})$.

\begin{lemma}\label{point}
Let $f:X\dashrightarrow Y$ be a rational morphism of smooth complete varieties and let $x\in X$ be a rational point such that
$f(x)$ is defined. Then $f_*([x])=[f(x)]$ in $\CH_0(Y)$.
\end{lemma}
\begin{proof}
Let $\Gamma\subset X\times Y$ be the graph of $f$. The preimage of $\{x\}\times\Gamma$ under the morphism
$\Delta_X\times\id_Y:X\times Y\to X\times X\times Y$ is the reduced scheme $\{x\}\times\{f(x)\}$. It follows from
\cite[Corollary 57.20]{EKM} that $(\Delta_X\times\id_Y)^*([x]\times[\Gamma])=[x]\times[f(x)]$ and hence
$f_*([x])=q_*([x]\times[f(x)])=[f(x)]$, where $q:X\times Y\to Y$ is the projection.
\end{proof}

\begin{cor}\label{composition}
Let $f:X\dashrightarrow Y$ and $g:Y\dashrightarrow Z$ be composable rational morphisms of smooth complete varieties and let $h:X\dashrightarrow Z$
be the composition of $f$ and $g$. Then $g\circ f=h$ in $\Mor_{\RatCor}(X,Z)$.
\end{cor}

\begin{proof}
Let $y$ be the rational point of $Y_{F(X)}$ corresponding to $f$.
By assumption, the rational morphism $g_{F(X)}:Y_{F(X)}\dashrightarrow
Z_{F(X)}$ is defined at $y$. By Lemma \ref{point}, the composition of correspondences $f$ and $g$ takes
$[y]$ to $[g_{F(X)}(y)]\in \CH_0(Z_{F(X)})$. Note that the latter class
is given by $h$.
\end{proof}

\begin{cor}
For every two composable rational morphisms $f:X\dashrightarrow Y$ and $g:Y\dashrightarrow Z$ of smooth complete varieties, we have
$(g\circ f)_*=g_*\circ f_*$ and $(g\circ f)^*=f^*\circ g^*$.
\end{cor}

\begin{cor}[{cf. \cite[Corollary 12.10]{MR1418952}}]
\label{bir inv}
The groups $A_0(X;M)$ and $A^0(X;M)$ are birational invariants of
the smooth complete variety $X$.
\end{cor}

\begin{prop}\label{trans}
Let
\[
\xymatrix{X' \ar@{-->}[d]_{\beta}\ar@{-->}[r]^{\alpha}  & X \ar@{-->}[d]^{\gamma} \\
Y' \ar@{-->}[r]^{\delta}  & Y }
\]
be a commutative diagram of dominant rational morphisms of smooth complete irreducible varieties
with $\dim (X)=\dim (Y)$ and $\dim (X')=\dim (Y')$. Suppose that the natural ring homomorphism $F(X)\tens_{F(Y)} F(Y')\to F(X')$
is an isomorphism. Then $\gamma^t\circ\delta=\alpha\circ\beta^t$.
\end{prop}

\begin{proof}
The generic fibers of the dominant rational morphisms $\gamma$ and $\beta$ are the single point schemes $\{x\}$ and $\{x'\}$
respectively. We have the following diagram:
\[
\xymatrix{\{x'\} \ar@{^{(}->}[d] \ar@{=}[r] & \{x'\}  \ar@{^{(}->}[d]\ar[r] & \{x\}
\ar@{^{(}->}[d] \\
X'_{F(Y')} \ar[d] \ar[r] & X_{F(Y')} \ar[d] \ar[r] &  X_{F(Y)} \ar[d]  \\
\Spec F(Y') \ar@{=}[r] &   \Spec F(Y') \ar[r] & \Spec F(Y).  }
\]
Note that as schemes, $\{x\}=\Spec F(X)$ and $\{x'\}=\Spec F(X')$. It follows from the assumption that the right
part of the diagram is cartesian and hence so is the top right square. In particular, $\{x'\}$ is closed in both
$X'_{F(Y')}$ and $X_{F(Y')}$.

By \cite[Proposition 62.4(2)]{EKM}, the composition  $\gamma^t\circ\delta$ is equal to the image of $[x]$ under the pull-back
homomorphism $\CH_0(X_{F(Y)})\to \CH_0(X_{F(Y')})$ and hence is equal to $[x']$ as $\{x'\}$ is the fiber product of $\{x\}$
and $X_{F(Y')}$ over $X_{F(Y)}$.

The composition $\alpha\circ\beta^t$ is equal to the image of $[x']$ under the push-forward homomorphism
$\alpha_*:\CH_0(X'_{F(Y')})\to \CH_0(X_{F(Y')})$. The rational map $\alpha_{F(Y')}:X'_{F(Y')}\to X_{F(Y')}$ is defined at
$x'$ and $\alpha_{F(Y')}(x')=x'$ as $x'$ is the closed point in both $X'_{F(Y')}$ and $X_{F(Y')}$. It follows
from Lemma \ref{point} that $\alpha_*([x'])=[x']$
in $\CH_0(X_{F(Y')})$.
\end{proof}

}

{


\renewcommand{\(}{\bigl(}
\renewcommand{\)}{\bigr)}


\newcommand{\m}{^{\times}}
\newcommand{\ff}{F^{\times}}
\newcommand{\fs}{F^{\times 2}}
\newcommand{\llg}{\longrightarrow}
\newcommand{\tens}{\otimes}
\newcommand{\inv}{^{-1}}


\newcommand{\Dfn}{\stackrel{\mathrm{def}}{=}}
\newcommand{\iso}{\stackrel{\sim}{\to}}


\newcommand{\mult}{\mathrm{mult}}
\newcommand{\sep}{\mathrm{sep}}
\newcommand{\id}{\mathrm{id}}
\newcommand{\diag}{\mathrm{diag}}
\newcommand{\op}{^{\mathrm{op}}}


\newcommand{\ihom}{\mathcal{H}om}
\newcommand{\CH}{\operatorname{CH}}
\newcommand{\F}{\operatorname{F}}
\renewcommand{\Im}{\operatorname{Im}}
\newcommand{\Ad}{\operatorname{Ad}}
\newcommand{\NN}{\operatorname{N}}
\newcommand{\Ker}{\operatorname{Ker}}
\newcommand{\Pic}{\operatorname{Pic}}
\newcommand{\Tor}{\operatorname{Tor}}
\newcommand{\Lie}{\operatorname{Lie}}
\newcommand{\ind}{\operatorname{ind}}
\newcommand{\ch}{\operatorname{char}}
\newcommand{\Inv}{\operatorname{Inv}}
\newcommand{\coker}{\operatorname{Coker}}
\newcommand{\res}{\operatorname{res}}
\newcommand{\Br}{\operatorname{Br}}
\newcommand{\Spec}{\operatorname{Spec}}
\newcommand{\SK}{\operatorname{SK}}
\newcommand{\Gal}{\operatorname{Gal}}
\newcommand{\SL}{\operatorname{SL}}
\newcommand{\GL}{\operatorname{GL}}
\newcommand{\gSL}{\operatorname{\mathbf{SL}}}
\newcommand{\DM}{\operatorname{\mathbf{DM}}}
\newcommand{\gO}{\operatorname{\mathbf{O}}}
\newcommand{\gGL}{\operatorname{\mathbf{GL}}}
\newcommand{\gPGU}{\operatorname{\mathbf{PGU}}}
\newcommand{\gSpin}{\operatorname{\mathbf{Spin}}}
\newcommand{\End}{\operatorname{End}}
\newcommand{\Hom}{\operatorname{Hom}}
\newcommand{\Aut}{\operatorname{Aut}}
\newcommand{\Mor}{\operatorname{Mor}}
\newcommand{\Map}{\operatorname{Map}}
\newcommand{\Sym}{\operatorname{S}}
\newcommand{\Ext}{\operatorname{Ext}}
\newcommand{\Nrd}{\operatorname{Nrd}}
\newcommand{\ord}{\operatorname{ord}}
\newcommand{\ra}{\rightarrow}
\newcommand{\xra}{\xrightarrow}


\newcommand{\A}{\mathbb{A}}
\renewcommand{\P}{\mathbb{P}}
\newcommand{\p}{\mathbb{P}}
\newcommand{\Z}{\mathbb{Z}}
\newcommand{\z}{\mathbb{Z}}
\newcommand{\N}{\mathbb{N}}
\newcommand{\Q}{\mathbb{Q}}
\newcommand{\QZ}{\mathop{\mathbb{Q}/\mathbb{Z}}}
\newcommand{\gm}{\mathbb{G}_m}
\newcommand{\hh}{\mathbb{H}}

\newcommand{\cA}{\mathcal A}
\newcommand{\cB}{\mathcal B}
\newcommand{\cF}{\mathcal F}
\renewcommand{\cH}{\mathcal H}
\newcommand{\cJ}{\mathcal J}
\newcommand{\cM}{\mathcal M}
\newcommand{\cN}{\mathcal N}
\newcommand{\cO}{\mathcal O}
\renewcommand{\cR}{\mathcal R}
\newcommand{\cS}{\mathcal S}
\newcommand{\cX}{\mathcal X}
\newcommand{\cU}{\mathcal U}
\renewcommand{\cL}{\mathcal L}


\newcommand{\falg}{F\mbox{-}\mathfrak{alg}}
\newcommand{\fgroups}{F\mbox{-}\mathfrak{groups}}
\newcommand{\fields}{F\mbox{-}\mathfrak{fields}}
\newcommand{\groups}{\mathfrak{Groups}}
\newcommand{\abelian}{\mathfrak{Ab}}

\appendix
\addtocounter{section}{12}
\renewcommand{\thesection}{RM}

\section
{Chow groups of Rost motives}
\label
{Chow groups of Rost motives}

We assume that $\ch F=0$ here.

\subsection{The binary motive}

Let $n$ be a positive integer, $p$ a prime integer and $s$ a symbol in $H^{n+1}(F,\mu_p^{\tens n})$. Set
\begin{align*}
b&=(p^{n}-1)/(p-1)=1+p+\dots+p^{n-1},\\
c&=(p^{n+1}-1)/(p-1)=1+p+\dots+p^{n}=bp+1=b+p^{n},\\
d&=p^{n}-1=b(p-1)=c-b-1.
\end{align*}

Let $\cX$ be the object in the triangulated category of motivic complexes $\DM(F,\Z)$
given by the simplicial scheme of a norm variety of $s$.
(Sometimes we will write $\cX$ as well for the corresponding object in
$\DM(F,\Z_{(p)})$.)
Write $Q_i$ for the Milnor operation
in the motivic cohomology of bidegree $(2p^i-1,p^i-1)$ (see \cite{MR2811603}) and set
\[
\mu=(Q_1\circ Q_2\circ\dots\circ   Q_{n-1})(\delta)\in H^{2b+1,b}(\cX,\Z)
\]
and
\[
\gamma=(Q_1\circ Q_2\circ\dots\circ  Q_{n-1}\circ Q_{n})(\delta)=\pm Q_{n-1}(\mu)\in
H^{2c,c-1}(\cX,\Z),
\]
where $\delta\in H^{n+2,n}(\cX,\Z)$ is the element
corresponding to the symbol $s$ (see \cite{MR2645334}).

The \emph{binary motive $\cM$ of $s$} is defined by the exact triangle
\[
\cX(b)[2b]\xra{x} \cM\xra{y} \cX\xra{\mu}\cX(b)[2b+1]
\]
in $\DM(F,\Z)$.

\subsection
{Symmetric powers}

As in \cite{MR2811603}, consider the symmetric powers $\Sym^i(\cM)$ for $i=0,1,\dots, p-1$
in $\DM(F,\Z_{(p)})$ of the binary motive $\cM$. There are the morphisms
\[
a_i:\Sym^i(\cM)\to \Sym^{i-1}(\cM)\tens \cM \quad\text{and}\quad b_i:\Sym^{i-1}(\cM)\tens\cM\to\Sym^i(\cM),
\]
defined by
$$
a_i(m_1\dots m_i)=\sum_{j=1}^{i}m_1\dots \hat m_j\dots m_i\tens m_j\;\text{ and }\;
b_i(m_1\dots m_{i-1}\tens m)=
m_1\dots m_{i-1}m.
$$

Consider the compositions:
\begin{align*}
x_i:&\ \Sym^{i-1}(\cM)(b)[2b]\xra{1\tens x}\Sym^{i-1}(\cM)\tens \cM\xra{b_i}\Sym^{i}(\cM),\\
y_i:&\ \Sym^{i}(\cM)\xra{a_i}\Sym^{i-1}(\cM)\tens \cM\xra{1\tens y}\Sym^{i-1}(\cM).
\end{align*}
We have $x_1=x$ and $y_1=y$. Set
\[
r_i=\Sym^i(y):\Sym^i(\cM)\to\cX,
\]
so $r_1=y$.

The following lemma can be checked by a direct computation:
\begin{lemma}\label{somesome}
For every $i=2,\dots, p-1$,

\noindent {\rm (1)} $y_i\circ b_{i-1}-b_{i-2}\circ (y_{i-1}\tens \id_{\cM})=\id_{\Sym^{i-1}(\cM)}\tens y$,

\noindent {\rm (2)} $r_{i-1}\circ y_i=i\cdot r_i$,

\noindent {\rm (3)} $y_1y_2\dots y_i=i!\cdot r_i$.
\end{lemma}

\begin{corollary}\label{manyi}
The diagram
\[
\begin{CD}
\Sym^{i-1}(\cM)(b)[2b]   @>{x_i}>> \Sym^{i}(\cM)   @>{r_i}>> \cX  \\
@V{y_{i-1}(b)[2b]}VV  @V{y_i}VV    @VV{\cdot i}V   \\
\Sym^{i-2}(\cM)(b)[2b]   @>{x_{i-1}}>> \Sym^{i-1}(\cM)   @>{r_{i-1}}>> \cX
\end{CD}
\]
is commutative.
\end{corollary}

Consider the following objects
$\cS=\Sym^{p-2}(\cM)$ and $\cR=\Sym^{p-1}(\cM)$ in $\DM(F,\Z_{(p)})$.
By \cite[\S5--\S6]{MR2811603},
the motive $\cR=\cR_s$ is isomorphic to a Chow motive living on a norm
variety of $s$ (this is the only place where we need characteristic $0$).
Since over any field extension of $F$ killing the symbol,
the element $\delta$ is trivial, the element $\mu$ is also trivial so that
the motive $\cM$ is isomorphic to $\Z_{(p)}\oplus\Z_{(p)}(b)$ and
the motive $\cR$
is isomorphic to the direct sum
$\Z_{(p)}\oplus\Z_{(p)}(b)\oplus\dots\oplus\Z_{(p)}(d)$.
It follows that $\cR$ is the
{\em Rost motive} of the symbol $s$ (as defined in Section \ref{Generic splitting
varieties}).

Consider the morphism $s=\frac{1}{(p-2)!}y_2\dots y_{p-1}:\cR\to \cM$. Taking the compositions of the
diagrams in Corollary \ref{manyi} and dividing out $(p-2)!$, we have:

\begin{lemma}\label{ccom}
The diagram
\[
\begin{CD}
\cS(b)[2b]   @>{x_{p-1}}>> \cR   @>{r_{p-1}}>> \cX  \\
@V{r_{p-2}(b)[2b]}VV  @V{s}VV    @VV{\cdot (p-1)}V   \\
\cX(b)[2b]   @>{x}>> \cM   @>{y}>> \cX
\end{CD}
\]
is commutative.
\end{lemma}

There are exact triangles \cite[(5.5) and (5.6)]{MR2811603} in
$\DM(F,\Z_{(p)})$:
\begin{equation}\label{first}
\cX(d)[2d]\xra{\Sym^{p-1}(x)} \cR\xra{y_{p-1}} \cS\ra \cX(d)[2d+1],
\end{equation}
\begin{equation}\label{second}
\cS(b)[2b]\xra{x_{p-1}} \cR\xra{r_{p-1}} \cX\ra\cS(b)[2b+1].
\end{equation}

\subsection
{Chow groups of Rost motives}

For $m\in\Z$, let $K_m^s(F)$ be the factor group of Milnor's $K$-group $K_m(F)$ by the subgroup generated by
the images of the norm homomorphisms
$K_m(L)\ra K_m(F)$
over all finite field extensions $L/F$ such that $s$ vanishes over $L$.

By \cite[Theorem 1.15]{MR2645334}, a nontrivial element $\alpha$
of the motivic cohomology group
$$
H^{i,j}(\cX,\Z):=H^i\big(\cX,\Z(j)\big)
$$
with $i> j$ can be uniquely written in
the form
\[
\alpha=x\gamma^k (Q_1^{\varepsilon_1}\circ Q_2^{\varepsilon_2}\circ \cdots\circ
Q_{n}^{\varepsilon_{n}})(\delta)=x\gamma^k Q^{\varepsilon}(\delta),
\]
where $x\in K_m^s(F)$ and $k$, $\varepsilon_i$ are integers such that
$k\geq 0$ and $\varepsilon_i=0$ or $1$. We have
\[
j=m+(c-1)k+\sum\varepsilon_k(p^k-1)+n,
\]
\begin{equation}\label{ququ}
w(\alpha)=:2j-i=m-2k-|\varepsilon|+(n-2),
\end{equation}
where $|\varepsilon|=\sum\varepsilon_k$.

Note that if $j\leq d$, then $k=0$ and $\varepsilon_{n}=0$.

\begin{lemma}\label{chowplus}
Let $0\leq j\leq d$. Then
\[
H^{2j+1,j}(\cX,\Z)=
\left\{
  \begin{array}{ll}
    (\Z/p\Z)\mu, & \hbox{if $j=b$;} \\
    0, & \hbox{otherwise.} \\
 \end{array}
\right.
\]
\end{lemma}

\begin{proof}
Let $\alpha=x\cdot Q^\varepsilon(\delta)\in H^{2j+1,j}(\cX,\Z)$, where $x\in K_m^s(F)$.
Recall that $\varepsilon_{n}=0$, hence
$|\varepsilon|\leq n-1$. By (\ref{ququ}),
\[
-1=w(\alpha)=m-|\varepsilon|+(n-2)
\]
and therefore,
\[
n-1\geq |\varepsilon| =m+(n-1).
\]
It follows that $m=0$ and $Q^\varepsilon(\delta)=
(Q_1\circ Q_2\circ\dots\circ   Q_{n-1})(\delta)=\mu$.
\end{proof}

For every $i=1,2,\dots, n-1$ set
\[
\widetilde Q_i=Q_1\circ\cdots\circ\widehat Q_i\circ\cdots\circ Q_{n-1}.
\]
Note that $\widetilde Q_i(\delta)\in H^{2(b-p^i+1),b-p^i+1}(\cX,\Z)$.

\begin{lemma}\label{ch}
Let $0\leq j\leq d$. Then
\[
H^{2j,j}(\cX,\Z)=
\left\{
  \begin{array}{llll}
    \Z\cdot 1, & \hbox{if $j=0$;} \\
    (\Z/p\Z)\widetilde Q_i(\delta), & \hbox{if $j=b-p^i+1$ and $1\leq i\leq n-1$;} \\
    K_1^s(F)\mu, & \hbox{if $j=b+1$;} \\
    0, & \hbox{otherwise.} \\
 \end{array}
\right.
\]
\end{lemma}

\begin{proof}
We may assume that $j>0$.
Let $\alpha=x Q^\varepsilon(\delta)\in H^{2j,j}(\cX,\Z)$, where $x\in K_m^s(F)$.
Recall that $|\varepsilon|\leq n-1$. By (\ref{ququ}),
\[
0=w(\alpha)=m-|\varepsilon|+(n-2)
\]
and therefore,
\[
n-1\geq |\varepsilon| =m+(n-2).
\]
It follows that $m\leq 1$. If $m=0$, we have $|\varepsilon|=n-2$, hence $Q^\varepsilon=\widetilde Q_i$
for $i=1,2,\dots, n-1$ and $\alpha\in (\Z/p\Z)\widetilde Q_i(\delta)$.

If $m=1$, then $|\varepsilon|=n-1$ and $Q^\varepsilon=Q_1\circ Q_2\circ\dots\circ   Q_{n-1}$, hence
$\alpha\in K_1^s(F)\mu$.
\end{proof}

\begin{lemma}\label{rs}
The canonical map
\[
H^{i,j}(\cS,\Z_{(p)})\to H^{i,j}(\cR,\Z_{(p)})
\]
is an isomorphism if $i<2d$ and $j<d$.
\end{lemma}

\begin{proof}
Use the triangle (\ref{first}).
\end{proof}

The Chow groups with coefficients in $\Z_{(p)}$ of a motive $\cN$ in
$\DM(F,\Z_{(p)})$ are defined as
$$
\CH^i(\cN):=H^{2i,i}(\cN,\Z_{(p)}).
$$

\begin{theorem}
\label{main}
Let $\cR$ be the Rost motive of a nontrivial $(n+1)$-symbol modulo $p$. Then
\[
\CH^j(\cR)=
\left\{
  \begin{array}{llll}
    \Z_{(p)}, & \hbox{if $j=0$;} \\
    p\Z_{(p)}, & \hbox{if $j=bk$, $1\leq k\leq p-1$;} \\
\Z/p\Z, & \hbox{if $j=bk-p^i+1$, $1\leq k\leq p-1$, $1\leq i\leq n-1$;} \\
    0, & \hbox{otherwise.} \\
 \end{array}
\right.
\]
\end{theorem}

\begin{proof}
We induct on $j=0,1,\dots, d$. First suppose that $j<b$. The triangle (\ref{second}) yields an isomorphism
\[
\CH^j(\cR)\simeq \CH^j(\cX)
\]
and the statement follows from Lemma \ref{ch}.

Now consider the case $j=b$. The triangle (\ref{second}) yields an exact sequence
\[
\CH^b(\cX)\to \CH^b(\cR) \to \CH^0(\cS)\to H^{2b+1,b}(\cX,\Z_{(p)}).
\]
The first term is trivial by Lemma \ref{ch} and the last is equal to $(\Z/p\Z)\mu$ by Lemma \ref{chowplus}.
By Lemma \ref{rs}, $\CH^0(\cS)=\CH^0(\cR)=\Z_{(p)}$.
The last map is the multiplication by $(p-1)\mu=-\mu$  by Lemma \ref{ccom}, hence $\CH^b(\cR)=p\Z_{(p)}$.

Now assume that $j=b+1$. The triangle (\ref{second}) gives an exact sequence
\[
H^{1,1}(\cS,\Z_{(p)}) \to \CH^{b+1}(\cX)\to \CH^{b+1}(\cR) \to \CH^1(\cS).
\]
By Lemma \ref{rs},
the last term is trivial and the
first term is equal to
$$
H^{1,1}(\cR,\Z_{(p)})\simeq H^{1,1}(\cX,\Z_{(p)})\simeq K_1(F)\otimes\Z_{(p)}.
$$
In view of Lemma \ref{ch}, the second term is
equal to $K_1^s(F)\mu$. The first map is multiplication by $\mu$, hence is surjective.
It follows that $\CH^{b+1}(\cR)=0$.

Now suppose that $j>b+1$. The triangle (\ref{second}) gives an exact sequence
\[
\CH^j(\cX)\to \CH^j(\cR) \to \CH^{j-b}(\cS)\to H^{2j+1,j}(\cX,\Z_{(p)}).
\]
It follows from Lemma \ref{rs} that $\CH^{j-b}(\cS)\simeq\CH^{j-b}(\cR)$.  By Lemmas \ref{chowplus} and \ref{ch},
the first and the last terms are trivial, hence
\[
\CH^j(\cR)\simeq \CH^{j-b}(\cR)
\]
and the result follows by induction.
\end{proof}

}

{

\newcommand{\s}{s}
\newcommand{\cM}{\mathcal{M}}
\newcommand{\Steen}{S}
\newcommand{\ISteen}{\mathbb{S}}
\newcommand{\sq}{\operatorname{sq}}
\newcommand{\st}{\operatorname{st}}

\newcommand{\onto}{\rightarrow\!\!\rightarrow}
\newcommand{\codim}{\operatorname{codim}}
\newcommand{\rdim}{\operatorname{rdim}}

\newcommand{\colim}{\operatorname{colim}}
\newcommand{\Div}{\operatorname{div}}
\newcommand{\Td}{\operatorname{Td}}
\newcommand{\td}{\operatorname{td}}
\newcommand{\itd}{\operatorname{itd}}

\newcommand{\ch}{\mathop{\mathrm{ch}}\nolimits}
\newcommand{\fch}{\mathop{\mathfrak{ch}}\nolimits}

\newcommand{\trdeg}{\operatorname{tr.deg}}

\newcommand{\SB}{X}

\newcommand{\Br}{\mathop{\mathrm{Br}}}
\newcommand{\Gal}{\mathop{\mathrm{Gal}}}
\renewcommand{\Im}{\mathop{\mathrm{Im}}}
\newcommand{\Pic}{\mathop{\mathrm{Pic}}}
\newcommand{\ind}{\mathop{\mathrm{ind}}}
\newcommand{\coind}{\mathop{\mathrm{coind}}}

\newcommand{\rk}{\mathop{\mathrm{rk}}}
\newcommand{\CH}{\mathop{\mathrm{CH}}\nolimits}

\newcommand{\BCH}{\mathop{\overline{\mathrm{CH}}}\nolimits}

\newcommand{\SO}{\operatorname{\mathrm{SO}}}
\newcommand{\OO}{\operatorname{\mathrm{O}}}

\newcommand{\Sp}{\operatorname{\mathrm{Sp}}}

\newcommand{\PGSp}{\operatorname{\mathrm{PGSp}}}

\newcommand{\Spin}{\operatorname{\mathrm{Spin}}}

\newcommand{\PGO}{\operatorname{\mathrm{PGO}}}

\newcommand{\PGL}{\operatorname{\mathrm{PGL}}}

\newcommand{\GL}{\operatorname{\mathrm{GL}}}

\newcommand{\SL}{\operatorname{\mathrm{SL}}}

\newcommand{\Ch}{\mathop{\mathrm{Ch}}\nolimits}
\newcommand{\BCh}{\mathop{\overline{\mathrm{Ch}}}\nolimits}

\newcommand{\TCh}{\mathop{\tilde{\mathrm{Ch}}}\nolimits}

\newcommand{\IBCH}{\mathop{\bar{\mathrm{CH}}}\nolimits}
\newcommand{\BCHE}{\mathop{\mathrm{Che}}\nolimits}

\newcommand{\Stab}{\mathop{\mathrm{Stab}}\nolimits}

\newcommand{\res}{\mathop{\mathrm{res}}\nolimits}
\newcommand{\cores}{\mathop{\mathrm{cor}}\nolimits}

\newcommand{\cd}{\mathop{\mathrm{cd}}\nolimits}
\newcommand{\Gcd}{\mathop{\mathfrak{cd}}\nolimits}

\newcommand{\pr}{\operatorname{\mathit{pr}}}

\newcommand{\inc}{\operatorname{\mathit{in}}}

\newcommand{\mult}{\operatorname{mult}}

\newcommand{\Char}{\mathop{\mathrm{char}}\nolimits}
\newcommand{\Dim}{\mathop{\mathrm{Dim}}\nolimits}
\newcommand{\id}{\mathrm{id}}
\newcommand{\coker}{\mathrm{coker}}
\newcommand{\Z}{\mathbb{Z}}
\newcommand{\F}{\mathbb{F}}
\newcommand{\A}{\mathbb{A}}
\newcommand{\PP}{\mathbb{P}}
\newcommand{\Q}{\mathbb{Q}}
\newcommand{\BBF}{\mathbb{F}}
\newcommand{\HH}{\mathbb{H}}
\newcommand{\C}{\mathbb{C}}
\newcommand{\R}{\mathbb{R}}
\newcommand{\Gm}{\mathbb{G}_{\mathrm{m}}}
\newcommand{\Ga}{\mathbb{G}_{\mathrm{a}}}
\newcommand{\cF}{\mathcal F}
\newcommand{\cO}{\mathcal O}
\newcommand{\cE}{\mathcal E}
\newcommand{\cT}{\mathcal T}
\newcommand{\cA}{\mathcal A}
\newcommand{\cB}{\mathcal B}

\newcommand{\Mor}{\operatorname{Mor}}
\newcommand{\Spec}{\operatorname{Spec}}
\newcommand{\End}{\operatorname{End}}
\newcommand{\Hom}{\operatorname{Hom}}
\newcommand{\Aut}{\operatorname{Aut}}

\newcommand{\Tors}{\operatorname{Tors}}

\newcommand{\an}{0}
\newcommand{\op}{\mathrm{op}}

\newcommand{\X}{\mathfrak{X}}
\newcommand{\Y}{\mathcal{Y}}

\newcommand{\pt}{\mathbf{pt}}
\newcommand{\PS}{\mathbb{P}}

\newcommand{\type}{\mathop{\mathrm{type}}}

\newcommand{\Coprod}{\operatornamewithlimits{\textstyle\coprod}}
\newcommand{\Prod}{\operatornamewithlimits{\textstyle\prod}}
\newcommand{\Sum}{\operatornamewithlimits{\textstyle\sum}}
\newcommand{\Oplus}{\operatornamewithlimits{\textstyle\bigoplus}}

\newcommand{\disc}{\operatorname{disc}}

\newcommand{\<}{\left<}
\renewcommand{\>}{\right>}
\renewcommand{\ll}{\<\!\<}
\newcommand{\rr}{\>\!\>}

\newcommand{\stiso}{\overset{st}\sim}
\newcommand{\miso}{\overset{m}\sim}

\newcommand{\Dfn}{\stackrel{\mathrm{def}}{=}}

\marginparwidth 2.5cm
\newcommand{\Label}{\label}

\newcommand{\compose}{\circ}
\newcommand{\ur}{\mathrm{ur}}
\newcommand{\Ker}{\operatorname{Ker}}
\newcommand{\TCH}{\Tors\CH}
\newcommand{\CM}{\operatorname{CM}}
\newcommand{\BCM}{\operatorname{\overline{CM}}}

\newcommand{\Ann}{\operatorname{Ann}}

\newcommand{\corr}{\rightsquigarrow}

\newcommand{\IW}{\mathfrak{i}_0}
\newcommand{\iw}{\mathfrak{i}}
\newcommand{\jw}{\mathfrak{j}}

\newcommand{\Hight}{\mathfrak{h}}
\newcommand{\Height}{\mathfrak{h}}
\newcommand{\Dh}{\mathfrak{d}}

\newcommand{\IS}{i_S}

\newcommand{\Sym}{\operatorname{Sym}}
\newcommand{\D}{D}

\newcommand{\BC}{\ast}
\newcommand{\NBC}{\compose}
\newcommand{\ABC}{\star}
\newcommand{\NC}{\compose}
\newcommand{\AC}{\bullet}

\newcommand{\Fields}{\mathbf{Fields}}
\newcommand{\fgFields}{\mathbf{fgFields}}
\newcommand{\Bin}{\mathbf{2^0}}
\newcommand{\Sets}{\mathbf{pSets}}

\newcommand{\eps}{\varepsilon}

\renewcommand{\phi}{\varphi}

\newcommand{\RatM}{\dashrightarrow}
\newcommand{\Place}{\dashrightarrow}

\newcommand{\WR}{\mathcal{R}}

\newcommand{\mf}{\mathfrak}

\appendix
\addtocounter{section}{18}
\renewcommand{\thesection}{SC}

\section
{Special correspondences}
\label{Special correspondences}

In this Appendix, $\CH$ is the Chow group with integer coefficients and
$\Ch$ is the Chow group with coefficients in $\F_p$.
The base field $F$ is of arbitrary characteristic $\ne p$.

Let $X$ be a smooth complete geometrically irreducible variety of
dimension $d:=p^n-1$ for some $n\geq1$.
A {\em special correspondence} $\sigma$ on $X$ is an
{\em anti-symmetric} ($\sigma^t=-\sigma$)
element of $\CH^b(X\times
X)$, where $b=(p^n-1)/(p-1)$,
such that for the image $H\in\CH^b(X_{F(X)})$ of $\sigma$ under the pull-back
along the morphism $X_{F(X)}\to X\times X$, induced by the generic point of the first
factor, one has:
\begin{enumerate}
\item
$\sigma_{F(X)}=1\times H-H\times 1$ and
\item
the degree of the $0$-cycle class $H^{p-1}$ is not divisible by $p$.
\end{enumerate}
(The original definition of a special correspondence given in \cite{markus}
is more restrictive, but we only need the above properties.)

As shown in \cite{markus}, any standard norm variety possesses a special
correspondence, and this explains our interest to varieties possessing
a special correspondence.

We are going to use the Steenrod operations on $\Ch$, \cite{MR1953530} or
\cite{Boisvert} (or \cite{MR2710781}).
For any $i\in\Z$, we write $\Steen^i$ for the cohomological Steenrod
operation which increases the codimension by $i$.
The way of indexing differs from that of \cite{MR1953530}.
In our indexing we have $\Steen^i=0$ if $i$ is not divisible by $p-1$.
Note that for existence of the Steenrod operations,
we do not need to assume quasi-projectivity of varieties,
\cite[\S10]{MR1953530}.

\subsection
{Rationality of Steenrod operations}
\label{Rationality of Steenrod operations}

Here is the main result of this subsection which we prove using the modification
due to R. Fino \cite{Fino}
of the original technique due to A. Vishik \cite{MR2305424}.
It extends (a weakened version of) Theorem \ref{mainZp}:
for $s=0$ the result below is very close to Theorem \ref{mainZp} (for $\Lambda=\F_p$)
weakened by the presence of an exponent $p$ element in the statement as well as by the requirement that
$Y$ is smooth (which we need for the Steenrod operations to be defined).
Note that unlike Theorem \ref{mainZp}, the proof of the result below does not
rely on Appendix \ref{Chow groups of Rost motives}.

We recall that two smooth complete irreducible varieties are {\em
equivalent} if there exist multiplicity $1$ correspondences (with
$\Lambda=\F_p$) between them in both directions.

\begin{thm}
\label{mainSI}
Let $X$ be an $A$-trivial (for $\Lambda=\F_p$) $F$-variety equivalent to an $A$-trivial
$F$-variety of dimension $p^n-1$ possessing a special correspondence.
Then for any smooth
irreducible
$F$-variety $Y$, any $m,s\in\Z$ with
$s>(m-b)(p-1)$, and any $y\in\Ch^m(Y_{F(X)})$, the element $\Steen^s(y)\in\Ch^{m+s}(Y_{F(X)})$
is {\em rational} (i.e., comes from $F$) up to
the class modulo $p$ of an exponent $p$ element of $\CH^{m+s}(Y_{F(X)})$.
\end{thm}

\begin{example}
Let $X$ be the Severi-Brauer variety of a degree $p$ central simple
$F$-algebra.
The variety $X$ has dimension $p-1$, is $A$-trivial (see Example \ref{PHV}), and possesses a
special correspondence (see \cite[Remark 7.17]{MR1776119}).
It follows that for any smooth $F$-variety $Y$ and any element
$y\in\Ch(Y_{F(X)})$, its $p$th power $y^p$ is rational up to the class modulo $p$ of an exponent
$p$ element.
Indeed, one may assume that $Y$ is irreducible and $y$ is homogeneous of
some codimension $m\geq 0$.
Then $y^p=\Steen^s(y)$ with $s=m(p-1)>(m-b)(p-1)$ (note that $n=1$ and $b=(p^n-1)/(p-1)=1$ here
so that Theorems \ref{mainZ} and \ref{mainZp}, if applicable at all, are vacuous in this situation).
\end{example}

\begin{proof}[Proof of Theorem \ref{mainSI}]
If $\deg\CH_0(X)\not\subset p\Z$, then $1\in\deg\Ch_0(X)$ and we are done
by Lemma \ref{surj}.
Below we are assuming that $\deg\CH_0(X)\subset p\Z$.

If the conclusion of Theorem \ref{mainSI} holds for an $A$-trivial variety $X$, then it also holds for
any $A$-trivial variety $X'$ equivalent to $X$.
Indeed, by Lemma \ref{surj}, the right and the bottom maps of the
commutative square
$$
\begin{CD}
\Ch(Y) @>>> \Ch(Y_{F(X)})\\
@VVV   @VVV\\
\Ch(Y_{F(X')}) @>>> \Ch(Y_{F(X\times X')})
\end{CD}
$$
are isomorphisms.
Therefore we may assume that the variety $X$ itself has dimension $d:=p^n-1$ and possesses a special
correspondence $\sigma\in\CH^b(X\times X)$.
As in our definition of special correspondence, let $H\in\CH^b(X_{F(X)})$ be the image of $\sigma$.

\begin{lemma}
\label{pH}
The element $pH\in\CH^b(X_{F(X)})\otimes\Z_{(p)}$ is rational.
\end{lemma}

\begin{proof}
We set $\rho:=\sigma^{p-1}$ (the power is taken using multiplication in the Chow group,
not composition of correspondences).
Then
\begin{equation}
\label{rhosigma}
(\sigma\compose\rho)_{F(X)}/\deg(H^{p-1})=1\times H+(p-1)H\times1,
\end{equation}
and the
pull-back of the rational element (\ref{rhosigma}) with respect to the diagonal of $X$ produces $pH$.
\end{proof}

The following lemma holds with coefficients in $\Z_{(p)}$ (cf.
\cite[Proposition 5.9]{markus}), although we need it now only for coefficients in $\F_p$:

\begin{lemma}
\label{Rost and special}
An abstract Rost motive (with coefficients in $\Z_{(p)}$ as well as with coefficients in $\F_p$) lives on $X$.
More precisely, there exists a symmetric (Rost) projector in $\CH^d(X\times
X)\otimes\Z_{(p)}$ such that over $F(X)$ it is equal to
$$
(1\times H^{p-1}+H\times H^{p-2}+\dots+H^{p-1}\times1)/\deg(H^{p-1}).
$$
\end{lemma}

\begin{proof}
It suffices to prove the statement for coefficients in $\Z_{(p)}$.

The (symmetric for $p\ne2$ and anti-symmetric for $p=2$) correspondence
$$
\rho:=\sigma^{p-1}\in\CH^d(X\times X)
$$
considered over $F(X)$ is congruent modulo $p$ to
the sum
\begin{equation}
\label{alt_sum}
1\times H^{p-1}+H\times H^{p-2}+\dots+H^{p-2}\times H+H^{p-1}\times 1.
\end{equation}
Let $c\in\Z_{(p)}$ be the inverse of the integer $\deg(H^{p-1})$.
The difference of the symmetric correspondence
$$
\rho':=(c\rho)\compose(c\rho)\in\CH^d(X\times X)\otimes\Z_{(p)},
$$
considered over $F(X)$, and the projector
$$
\pi:=c(1\times H^{p-1}+H\times H^{p-2}+\dots+ H^{p-1}\times1)\in\CH^d(X\times X)_{F(X)}\otimes\Z_{(p)}
$$
is a linear combination of $H^i\times H^{p-i-1}$, $i=0,1,\dots,p-1$
with divisible by $p$ coefficients.
The motive defined by $\pi$ is
$$
(X_{F(X)},\pi)\simeq\Z_{(p)}\oplus\Z_{(p)}(b)\oplus\dots\oplus\Z_{(p)}(d).
$$
Replacing $\rho'$ by $(\rho')^{\compose p^r}$ with sufficiently big $r$,
we keep the symmetry of $\rho'$ and get that the difference of $\rho'$ and
$\pi$ is
a linear combination of $H^i\times H^{p-i-1}$
with coefficients divisible by $p^{p-1}$.
It follows by Lemma \ref{pH} that there exists a symmetric correspondence
$\rho''\in\CH^d(X\times X)\otimes\Z_{(p)}$ with $\rho''_{F(X)}=\pi$.

Let $A$ (respectively, $B$) be the (commutative) subring of
the ring $\End M(X)$ (respectively, $\End M(X_{F(X)})$)
generated by $\rho''$ (respectively, $\pi$).
The kernel of the ring epimorphism $A\onto B$ consists of nilpotent
elements.
Indeed, any element of the kernel vanishes over $F(X)$ and, by
specialization, over the residue field of any point of $X$.
Therefore it is nilpotent by \cite[Theorem 67.1]{EKM}.
It follows by an argument like in \cite[Corollary 92.5]{EKM}
that there exists a projector in $A\subset\CH^d(X\times X)\otimes\Z_{(p)}$ whose image in $B$ is
$\pi$.
This projector is symmetric because $A$ consists of symmetric elements
only.
The motive given by this projector is an abstract Rost motive.
\end{proof}

Recall that $\deg\CH_0(X)\subset p\Z$ is assumed.

\begin{cor}
\label{ratF(X)rat}
For any $i>0$ and
any $\alpha\in\CH^i(X),\beta\in\CH_i(X_{F(X)})$, the degree of the $0$-cycle class
$\alpha_{F(X)}\cdot\beta$ is divisible by $p$.
\end{cor}

\begin{proof}
Combine Lemma \ref{Rost and special} with Corollary \ref{cor3.11}.
\end{proof}

\begin{cor}
\label{preasb}
For any $i,j>0$, $k,l\geq0$, $\alpha\in\CH^i(X)$ and $\beta\in\CH^j(X)$, the degree of the element
$(\alpha\times\beta)_{F(X)}\cdot(H^k\times H^l)\in\CH^{i+j+b(k+l)}(X\times X)_{F(X)}$ is divisible by $p^2$.
\end{cor}

\begin{proof}
The degree is equal to the product of the degrees $\deg(\alpha_{F(X)}\cdot H^k)$ and
$\deg(\beta_{F(X)}\cdot H^l)$ each of which is divisible by $p$ by
Corollary \ref{ratF(X)rat}.
\end{proof}

\begin{cor}
\label{asb}
For any $i,j>0$, $r\geq0$, $\alpha\in\CH^i(X)$ and $\beta\in\CH^j(X)$, the degree of the element
$(\alpha\times\beta)\cdot\sigma^r\in\CH^{i+j+br}(X\times X)$ is divisible by $p^2$.
\end{cor}

\begin{proof}
The element $\sigma^r_{F(X)}$ is a linear combination of $H^k\times H^l$
(with $k+l=r$).
\end{proof}

\begin{lemma}
\label{Hrat}
For any $i>0$, the element $\Steen^i(H)\in\Ch^{b+i}(X_{F(X)})$ is
rational.
\end{lemma}

\begin{proof}
We prove first that $\deg(H^j\Steen^i(H))\equiv0\pmod{p}$ for any
$j\geq0$.
Assume the contrary.
Then $i$ is divisible by $b$, say, $i=bk$ for some $k>0$, and
$j=p-2-k$.
Computing the composition $\Steen^i(\sigma)\compose\sigma^{p-1-k}\in\Ch(X\times X)$
over $F(X)$ we get a multiple with a nonzero coefficient $\in\F_p$ of $H\times 1\in\Ch^b(X\times
X)$.
Taking the pull-back with respect to the diagonal
shows that the class of $H$ modulo $p$ is rational and therefore $1\in\deg\Ch_0(X)$, a contradiction.

Now the composition $\Steen^i(\sigma)\compose\sigma^{p-1}\in\Ch(X\times
X)$ computed over $F(X)$ gives a multiple with a nonzero coefficient of
$1\times\Steen^i(H)$, and we finish the proof pulling back with respect to the diagonal of
$X$.
\end{proof}

We set $\rho:=\sigma^{p-1}\in\CH^d(X\times X)$.
Since $\rho_{F(X)}$ is congruent modulo $p$ to the alternating sum
(\ref{alt_sum}),
any element $x\in\Ch(X\times Y)$ of the form $x= x'\compose\rho$ for some
$x'\in\Ch(X\times Y)$ decomposes over $F(X)$ as
\begin{equation}
\label{x_F(X)}
x_{F(X)}=1\times x_0+H\times x_1+\dots+H^{p-1}\times x_{p-1}
\end{equation}
with some
$x_0,x_1,\dots,x_{p-1}\in\Ch(Y_{F(X)})$.
Note that $x_0$ coincides with the image of $x$ in $\Ch(Y_{F(X)})$.
The key statement in the proof of Theorem \ref{mainSI} is the following
Proposition
(where the $A$-triviality assumption is not needed):

\begin{prop}
\label{mainSIprop}
Let $X$ be a smooth complete irreducible variety of dimension $d=p^n-1$
possessing a special correspondence $\sigma\in\CH^b(X\times X)$.
Let $Y$ be a smooth irreducible variety and $x\in\Ch^m(X\times Y)$ an
element of the form $x=x'\compose\rho$.
Then for any $s>(m-b)(p-1)$ the element $\Steen^s(x_0)\in\Ch^{m+s}(Y_{F(X)})$ is rational
up to the class modulo $p$ of an exponent $p$ element.
\end{prop}

\begin{proof}
For any $x\in\Ch(X\times Y)$, we have the relation
$$
\pr_{2*}\sum_{0\leq i\leq d+s}b_i\cdot\Steen^{d+s-i}(x)=\Steen^{d+s}\big(\pr_{2*}(x)\big),
$$
where $b_i:=b_i(-T_X)\in\Ch^i(X)$ (this is rather $b_i(T_X)$ in notation of \cite{markus}),
where $b_i(\cdot)$ are the components of the multiplicative Chern class as defined in
\cite[\S6.1]{Boisvert}.
Note that the product in the expression $b_i\cdot\Steen^{d+s-i}(x)$ is the product of
the $\Ch(X)$-module $\Ch(X\times Y)$ so that the expression actually means
$(b_i\times Y)\cdot\Steen^{d-i}(x)$ (now in the sense of the product in the ring $\Ch(X\times Y)$).

Now we assume that $x\in\Ch^m(X\times Y)$ with $m$ such that $s>(m-b)(p-1)$.
In this case $\pr_{2*}(x)\in\Ch^{m-d}(Y)$ and $\Steen^{d+s}\big(\pr_{2*}(x)\big)=0$
because $d+s>s>(m-b)(p-1)\geq(m-d)(p-1)$.
Besides, $\Steen^{d+s}(x)=0$ because $d+s=b(p-1)+s>m(p-1)$.
Therefore we have
\begin{equation}
\label{starteq}
\pr_{2*}\sum_{0<i\leq d+s}b_i\cdot\Steen^{d+s-i}(x)=0.
\end{equation}

Putting $x\compose\rho$ in place of $x$ in relation (\ref{starteq})
and using the equality
$$
\Steen^\bullet(x\compose\rho)=\big(b_\bullet\cdot\Steen^\bullet(x)\big)\compose\Steen^\bullet(\rho),
$$
(together with the projection formula)
we rewrite the left part of relation (\ref{starteq}) as
\begin{multline*}
\pr_{2*}\sum_{\substack{i+j+k+l=d+s \\ i>0,\,j,\,k,\,l\geq0}}
b_i\cdot\Big(\big(b_j\Steen^k(x)\big)\compose\Steen^l(\rho)\Big)\\
\begin{aligned}
&=\pr_{2*}\sum\pr_{13*}\left(\Big(\big(b_i\cdot\Steen^l(\rho)\big)\times[Y]\Big)
\cdot\Big([X]\times\big(b_j\cdot\Steen^k(x)\big)\Big)\right)\\
&=\pr_{2*}\sum\pr_{23*}\left(\Big(\big(b_i\cdot\Steen^l(\rho)\big)\times[Y]\Big)
\cdot\Big([X]\times\big(b_j\cdot\Steen^k(x)\big)\Big)\right)\\
&=\pr_{2*}\sum \Big(b_j\cdot\Steen^k(x)\Big)\cdot
\left(\Big(\pr_{2*}\big(b_i\cdot\Steen^l(\rho)\big)\Big)\times[Y]\right)\\
&=\pr_{2*}\sum b_j\cdot \Big(\pr_{2*}\big(b_i\cdot\Steen^l(\rho)\big)\Big) \cdot
\Steen^k(x).
\end{aligned}
\end{multline*}
Therefore
\begin{equation}
\label{2eq}
\pr_{2*}\sum_{\substack{i+j+k+l=d+s \\ i>0,\,j,\,k,\,l\geq0}}
b_j\cdot \Big(\pr_{2*}\big(b_i\cdot\Steen^l(\rho)\big)\Big) \cdot \Steen^k(x)=0
\end{equation}
for any $x\in\Ch^m(X\times Y)$.

We recall that $\rho=\sigma^{p-1}$.
Therefore
$$
\Steen^l(\rho)=\Steen^l(\sigma^{p-1})=\sum_{l_1+\dots+l_{p-1}=l}
\Steen^{l_1}(\sigma)\cdot\ldots\cdot\Steen^{l_{p-1}}(\sigma).
$$
Relation (\ref{2eq}) rewrites as
\begin{equation*}
\pr_{2*}\sum_{\substack{i+j+k+l_1+\dots+l_{p-1}=d+s \\ i>0;\;j,\,k,\,l_1,\dots,l_{p-1}\geq0}}
b_j\cdot \Big(\pr_{2*}\big(b_i\cdot\Steen^{l_1}(\sigma)\cdot\ldots
\cdot\Steen^{l_{p-1}}(\sigma)\big)\Big) \cdot \Steen^k(x)=0.
\end{equation*}
Therefore, fixing for each integer $k\geq0$ an integral representative $\Steen^k_x\in\CH^{m+k}(X\times
Y)$ of $\Steen^k(x)\in\Ch^{m+k}(X\times Y)$
as well as an integral
representative $\Steen^k_{\sigma}\in\CH^{b+k}(X\times X)$
of $\Steen^k(\sigma)\in\Ch^{b+k}(X\times X)$
(where we choose $\sigma$ for $\Steen^0_\sigma$),
we get that the sum
\begin{equation}
\label{3eq}
\sum_{\substack{i+j+k+l_1+\dots+l_{p-1}=d+s \\ i>0;\;j,\,k,\,l_1,\dots,l_{p-1}\geq0}}
\pr_{2*}\left(
b_j\cdot \Big(\pr_{2*}\big(b_i\cdot\Steen^{l_1}_\sigma\cdot\ldots
\cdot\Steen^{l_{p-1}}_{\sigma}\big)\Big) \cdot \Steen^k_x\right)
\end{equation}
is divisible by $p$ in $\CH(Y)$.
Taking for $x$ an element of the form $x=x'\compose\rho$ with some $x'\in\Ch^m(X\times
Y)$ and passing over $F(X)$, we are going to show that the sum of (\ref{3eq})
is equal modulo $I:=p^2\CH(Y_{F(X)})+p\Im\big(\CH(Y)\to\CH(Y_{F(X)})\big)$
to the class of $\deg(b_d)\Steen^s_{x_0}$, where $\Steen^s_{x_0}\in\CH^{m+s}(Y_{F(X)})$
is an integral representative of $\Steen^s(x_0)$.
More precisely, we show that the summand for $i=d$ and $k=s$ modulo $I$ is
$\deg(b_d)\Steen^s_{x_0}$
(up to multiplication by a prime to $p$ integer)
while each other summand modulo $I$ is $0$.
Since the integer $\deg b_d$ is not divisible by $p^2$ (see \cite[Theorem 9.9]{markus}), we will
get
that $\Steen^s(x_0)$ is rational up to the classes modulo $p$ of an element of
exponent $p$.

For $i=d$ and $k=s$ we have $j=l_1=\dots=l_{p-1}=0$, and the corresponding summand
of (\ref{3eq}) is equal to
\begin{equation}
\label{jl0}
\pr_{2*}\left(
\pr_{2*}(b_d\cdot\rho) \cdot \Steen^s_x\right)_{F(X)}.
\end{equation}
Since $\deg(b_d)$ is divisible by $p$ and $\rho_{F(X)}$ is congruent modulo $p$ to
(\ref{alt_sum}), the factor $\pr_{2*}(b_d\cdot\rho)$
is congruent modulo $I$ to $\deg(b_d)\cdot H^{p-1}$.
Taking into account the decomposition (\ref{x_F(X)}) of $x_{F(X)}$, it follows that (\ref{jl0})
is congruent modulo $I$ to $\deg(b_d)\Steen^s_{x_0}$
up to multiplication by the prime to $p$ integer $\deg(H^{p-1})$.
Below we are assuming that $i\ne d$.

For $l=0$ (where $l:=l_1+\dots+l_{p-1}$), that is to say, for
$l_1=\dots=l_{p-1}=0$, an arbitrary summand we get is of the form
$$
\pr_{2*}\left(
b_j\cdot \pr_{2*}\big(b_i\cdot\rho\big) \cdot
\Steen^k_x\right)_{F(X)}
$$
(with $i+j+k=d+s$).
Note that $\pr_{2*}\big(b_i\cdot\rho\big)_{F(X)}=0$
if $i$ is not divisible by $b$.
Otherwise, since $i>0$ and $\rho_{F(X)}$ is congruent modulo $p$ to (\ref{alt_sum}),
$\pr_{2*}\big(b_i\cdot\rho\big)_{F(X)}$ is  by Corollary \ref{ratF(X)rat} congruent modulo $I$ to a multiple of
$pH^{i/b}$ so that we only need to show that
the element
\begin{equation}
\label{elm-modp}
\pr_{2*}\left(
b_j\cdot H^{i/b} \cdot
\Steen^k(x)\right)_{F(X)}\in\Ch(Y_{F(X)})
\end{equation}
is rational.

If $j>0$, then computing $\Steen^k(x)_{F(X)}=\Steen^k(x_{F(X)})$
via the decomposition (\ref{x_F(X)}) of $x_{F(X)}$  and using Corollary
\ref{ratF(X)rat}, we see that
the element (\ref{elm-modp}) is $0$.
Let us assume that $j=0$ and show that (\ref{elm-modp}) is $0$ as well.
It suffices to show this with $x$ replaced by an arbitrary summand of the decomposition
(\ref{x_F(X)}).
Putting $1\times x_0$ (the first summand of the decomposition) in place of $x$, we get
$$
\pr_{2*}\big(H^{i/b}\cdot\Steen^k(1\times x_0)\big)=
\pr_{2*}\big(H^{i/b}\times \Steen^k(x_0)\big)
$$
which is $0$ because $i\ne d$.
Putting any other summand $H^r\times x_r$ ($r\geq1$) of the decomposition,
we get a multiple of $\Steen^{s+rb}(x_r)$ which is $0$ because
$x_r\in\Ch^{m-rb}(Y_{F(X)})$ and $s+rb>s>(m-b)(p-1)\geq(m-rb)(p-1)$.

It remains to consider the case of $l>0$.
We have
$$
\Steen^l(\sigma)_{F(X)}=1\times\Steen^l(H)-\Steen^l(H)\times1
$$
and $\Steen^l(H)$ is rational by Lemma \ref{Hrat}.
Therefore
\begin{equation}
\label{decomp}
(\Steen^l_\sigma)_{F(X)}=p\theta_l+1\times\Steen^l_H-\Steen^l_H\times1
\end{equation}
for some $\theta_l\in\CH^{b+l}(X\times X)_{F(X)}$ and a {\em rational}
integral representative $\Steen^l_H$ of $\Steen^l(H)$.

Let us decompose as in (\ref{decomp}) every factor with positive superscript of the product
$$
(\Steen^{l_1}_\sigma\cdot\ldots\cdot\Steen^{l_{p-1}}_\sigma)_{F(X)}
$$
(appearing in (\ref{3eq})$_{F(X)}$),
expand the product and consider an arbitrary summand $P$ of the expansion.
We are going to show that the element
\begin{equation}
\label{Peq}
\pr_{2*}\Big(
b_j\cdot \pr_{2*}(b_i\cdot P) \cdot (\Steen^k_x)_{F(X)}\Big)
\end{equation}
modulo $I$ is $0$.

If $P$ contains the factor $p\theta_{?}$ (at least) two times,
then the result is divisible by $p^2$ so that (\ref{Peq}) is indeed $0$ modulo $I$.

Assume that the factor $p\theta_{?}$ is present precisely one time in $P$.
So, we already have divisibility by $p$ and it suffices to show that the
element
\begin{equation}
\label{elmnt}
\pr_{2*}\big(t\cdot\Steen^k(x_{F(X)})\big)\in\Ch^{m+s}(Y_{F(X)})
\end{equation}
is $0$
for an element $t\in\Ch^{i+j+l}(X_{F(X)})$ such that
$p\cdot t=b_j\cdot \pr_{2*}(b_i\cdot P)$.
Replacing $x_{F(X)}$ in (\ref{elmnt}) by an arbitrary summand of the decomposition
(\ref{x_F(X)}), we get $0$ always
(and for an arbitrary $t\in\Ch^{i+j+l}(X_{F(X)})$)
with only one possible exception: for
the summand $1\times x_0$, namely.

Putting $1\times x_0$ in place of $x_{F(X)}$ in (\ref{elmnt}), we may get
a nonzero result only if
$k=s$, that is, $i+j+l=d$.
In this case $t$ is a $0$-cycle class and the element (\ref{elmnt}) is divisible by its degree.
It suffices therefore to show that
the degree $\deg\Big(b_j\cdot\pr_{2*}(b_i\cdot P)\Big)$ is divisible by $p^2$.
If $j>0$, then the degree is divisible by $p^2$ by Corollary \ref{ratF(X)rat}
(recall that $P$ is divisible by $p$).
Therefore we may assume that $j=0$, that is $i+l=d$.
In this case $b_i\cdot P$ is a $0$-cycle class (on $X\times X$) and
the corresponding summand of (\ref{3eq}) is divisible by its degree.
But degree of $b_i\cdot P$ coincides with degree of
$b_i\cdot\pr_{1*}(P)$ which is divisible by $p^2$ by Corollary
\ref{ratF(X)rat} (we recall that $i>0$ and that $P$ is already divisible by
$p$).

At last, let us assume that $P$ contains no $\theta_?$ as a factor.
Then $P$ must contain at least one factor of the type $1\times\Steen^?_H$
(we call it a {\em second type factor} because it corresponds to the second summand of the decomposition
(\ref{decomp})) or of the type $\Steen^?_H\times1$ (a {\em third type factor}).
Moreover, any factor of $P$ is either $\sigma$ or $1\times\Steen^?_H$
(a second type factor)
or $\Steen^?_H\times1$ (a third type factor).
It follows by Corollary \ref{ratF(X)rat} that
$\pr_{2*}(b_i\cdot P)$ is divisible by $p$.
Therefore we may assume that $k=s$.
The element (\ref{Peq}) is then divisible by degree
of the $0$-cycle class
$(b_i\times b_j)\cdot P$.
This degree is divisible by $p^2$ if $j>0$ or if $P$ contains a factor of
the second type by Corollary \ref{asb}.

In the remaining case we have $j=0$, $i+l=d$, any factor of $P$ equals $\sigma$
or has the third type
with at least one factor of the third type.
Therefore $P=(\alpha\times1)\cdot\sigma^r$ with some $r<p-1$ and some
$\alpha\in\CH(X_{F(X)})$ (which is in fact rational but we do not care about this anymore).
It follows that the element $\pr_{2*}(b_i\cdot P)$ is a $0$-cycle class
and the element (\ref{Peq}) is divisible by its degree which is
$$
\deg\big(\pr_{2*}(b_i\cdot P)\big)=\deg\big(\pr_{1*}(b_i\cdot P)\big)=
\deg\big(b_i\cdot \pr_{1*}(P)\big),
$$
but already the element
$$
\pr_{1*}(P)=
\alpha\cdot\pr_{1*}(\sigma^r)
$$
is trivial because $\pr_{1*}(\sigma^r)=0$ for $r<p-1$.
\end{proof}

We finish now the proof of Theorem \ref{mainSI}.
Let $x$ be an element of $\Ch^m(X\times Y)$ mapped to
$y\in\Ch^m(Y_{F(X)})$.
Since $X$ is $A$-trivial, the element $(cx)\compose\rho$,
where $c\in\F_p$ is the inverse to the class modulo $p$ of the integer
$\deg(H^{p-1})$,
is also mapped to $y$ (see Lemma \ref{criter}).
Replacing $x$ by $(cx)\compose\rho$, we apply Proposition \ref{mainSIprop} to the
new $x$ getting the desired result.
\end{proof}

\subsection
{Generators of Chow groups of Rost motives}

In this appendix, we provide an elementary construction of homogeneous generators
of the Chow group of a Rost motive in the spirit of \cite{MR2015051}.
Here we have $\Lambda=\Z_{(p)}$ so that $\CH$ stands for the Chow group with
coefficients in $\Z_{(p)}$.

Let $X$ be a standard norm variety of dimension $d:=p^n-1$
($p$ a prime, $n\geq1$).
Let $\sigma\in\CH^b(X\times X)$, $b:=d/(p-1)$, be a special
correspondence on $X$, $H\in\CH^b(X_{F(X)})$ the image of $\sigma$.
Let $\rho\in\CH^d(X\times X)$ be a symmetric (Rost) projector on $X$ such that
$$
\rho_{F(X)}=(1\times H^{p-1}+H\times H^{p-2}+\dots+H^{p-1}\times1)/\deg(H^{p-1})
$$
(see Lemma \ref{Rost and special}).

Let $m$ be an integer satisfying $1\leq m\leq n-1$.
Assume that there exists a norm variety $Y$ of dimension $p^m-1$ with a morphism $f:Y\to
X$.
(This assumption is satisfied if the base field $F$ is $p$-special and has characteristic $0$ by
\cite[Corollary 1.22]{MR2220090}.)

\begin{prop}
\label{prop s.16}
For any $r$ with $1\leq r\leq p-1$, the element
$$
\alpha:=\rho_*(\sigma^r)_*f_*[Y]\in\CH_{p^m-1+(p-1-r)b}(X)
$$
is of order $p$.
\end{prop}

\begin{rem}
Since $\rho$ is symmetric, we have $\rho_*=\rho^*$
and therefore
$$
\alpha\in\rho^*\CH^{p^m-1+(p-1-r)b}(X)=\CH^{p^m-1+(p-1-r)b}(X,\rho).
$$
Since the Chow group is of order $p$ by Theorem \ref{main}, $\alpha$ is its generator.
By Theorem \ref{main} once again, varying $m$ and $r$, we get generators for the whole torsion
part of the Chow group $\CH^*(X,\rho)$ of the Rost motive.
\end{rem}

\begin{proof}[Proof of Proposition \ref{prop s.16}]
We first check that $p\alpha=0$.
Since $X$ has a closed point of degree $p$, it suffices to check that
$\alpha_{F(X)}=0$.
Over $F(X)$ we have
$$
\alpha_{F(X)}=(\beta)_*f_*[Y],
$$
where $\beta$ is a linear combination of $H^i\times H^{r-i}$ for
$i=0,1,\dots,r$.
This is $0$ because $\dim Y\ne \codim H^i=bi$ for any $i$.
Indeed,
$$
0<\dim Y< p^{n-1}\leq1+p+\dots+p^{n-1}=b.
$$

Now we show that $\alpha\ne0$ as follows:
$$
\alpha\ne0\;\;\Leftarrow\;\;
\alpha\!\!\mod{p}\ne0\;\;\Leftarrow\;\;
\Steen^{p^m-1}(\alpha)\ne0\;\;\Leftarrow\;\;
p^2\not|\deg\Big((\Steen^{p^m-1}_\alpha)_{F(X)}\cdot H^{p-1-r}\Big),
$$
where $\Steen^{p^m-1}_\alpha\in\CH_{(p-1-r)b}(X)$ is a
representative of the modulo $p$ cycle class $\Steen^{p^m-1}(\alpha)\in\Ch_{(p-1-r)b}(X)$.
The implication on the very right comes from Corollary \ref{ratF(X)rat}.
(The degree modulo $p^2$ does not depend on the choice of the integral representative by
Corollary
\ref{ratF(X)rat} once again.)

In order to compute
$\deg\Big((\Steen^{p^m-1}_\alpha)_{F(X)}\cdot H^{p-1-r}\Big)$
we use the formula
$$
\Steen^\bullet\big((\rho\compose\sigma^r)_*f_*[Y]\big)=
\big(\Steen^\bullet(\rho\compose\sigma^r)\big)_*\Steen_\bullet f_*[Y],
$$
where $S_\bullet$ is the total {\em homological} Steenrod
operation.\footnote{This formula is proved in\cite[Proposition 2.1]{Fino2} for $p=2$;
the proof can be easily adapted for arbitrary $p$.}
Since $\Steen_\bullet f_*[Y]=f_*(\Steen_\bullet[Y])=b_\bullet^Y$, where
$b_\bullet^Y:=f_*\big(b_\bullet(-T_Y)\big)$, the degree in question is
congruent modulo $p^2$ to
\begin{multline*}
\deg\sum_{\substack{i+j=p^m-1 \\
i,\,j\,\geq0}}H^{p-1-r}\cdot(\Steen^j_{\rho\compose\sigma^r})_*(b_i^Y)_{F(X)}\\[-1em]
\begin{aligned}
&=\sum\deg\Big(H^{p-1-r}\cdot\pr_{2*}(b_i^Y\cdot\Steen^j_{\rho\compose\sigma^r})_{F(X)}\Big)\\
&=\sum\deg\Big(\big((b_i^Y)_{F(X)}\times
H^{p-1-r}\big)\cdot(\Steen^j_{\rho\compose\sigma^r})_{F(X)}\Big)\\
&=\sum\deg\Big((b_i^Y)_{F(X)}\cdot\pr_{1*}\big((\Steen^j_{\rho\compose\sigma^r})_{F(X)}\cdot(1\times
H^{p-1-r})\big)\Big),
\end{aligned}
\end{multline*}
where $\Steen^j_{\rho\compose\sigma^r}$ are representatives of
$\Steen^j(\rho\compose\sigma^r)$.
Let us choose representatives $\Steen^j_{\sigma^r}$ of
$\Steen^j(\sigma^r)$.
Since $(\rho\compose\sigma^r)_{F(X)}=\sigma^r_{F(X)}$, the classes modulo $p$ of
$(\Steen^j_{\rho\compose\sigma^r})_{F(X)}$ and $(\Steen^j_{\sigma^r})_{F(X)}$
coincide.
It follows by Corollary \ref{ratF(X)rat} that we may remove $\rho$ from the
formula (the resulting degree modulo $p^2$ is not changed).

Taking $j=0$, we get the product $\deg b_{\dim Y}^Y\cdot\deg(H^{p-1})$ which is not divisible by $p^2$.
It remains to show that for any $j>0$ the degree
$$
\deg\Big((b_i^Y\times
H^{p-1-r})\cdot\Steen^j_{\sigma^r}\Big)=
\deg\Big(b_i^Y\cdot\pr_{1*}\big(\Steen^j_{\sigma^r}\cdot(1\times
H^{p-1-r})\big)\Big)
$$
(everything is over $F(X)$ although we omit the subscription $F(X)$)
is divisible by $p^2$.
We have
$$
\Steen(\sigma^r)=\sum_{j_1+\dots+j_r=j}\Steen^{j_1}(\sigma)\cdot\ldots\cdot
\Steen^{j_r}(\sigma).
$$

As in the end of Subsection \ref{Rationality of Steenrod operations},
for any $j>0$, we have
\begin{equation}
\label{dcmp}
\Steen^j_\sigma=p\theta_j+1\times\Steen^j_H-\Steen^j_H\times1
\end{equation}
for some $\theta_j\in\CH^{b+j}(X\times X)_{F(X)}$ and a rational
integral representative $\Steen^j_H$ of $\Steen^j(H)$.

Let us decompose as in (\ref{dcmp}) every factor with positive superscript of the product
$$
\Steen^{j_1}_\sigma\cdot\ldots\cdot\Steen^{j_r}_\sigma,
$$
expand the product and consider an arbitrary summand $P$ of the expansion.
We want to show that the degree
\begin{equation}
\label{degree}
\deg\Big((b_i^Y\times
H^{p-1-r})\cdot P\Big)=
\deg\Big(b_i^Y\cdot\pr_{1*}\big(P\cdot(1\times
H^{p-1-r})\big)\Big)
\end{equation}
is divisible by $p^2$.

If $P$ contains a factor of the type $p\theta_{?}$ (at least) one time,
then the result is divisible by $p^2$ by Corollary \ref{ratF(X)rat} applied to the right-hand side presentation
of the degree in (\ref{degree}).

Let us assume that $P$ contains no $\theta_?$ as a factor.
Then $P$ must contain at least one factor of the second or of the third
type.
Moreover, any factor of $P$ is either $\sigma$ or $1\times\Steen^?_H$
(a second type factor) or $\Steen^?_H\times1$ (a third type factor).
If a factor of the second type is present, the degree is divisible by $p^2$
by Corollary \ref{preasb} applied to the left-hand side presentation
of the degree in (\ref{degree}).
If there is no factor of the second type, then already
$\pr_{1*}\Big(P\cdot(1\times H^{p-1-r})\Big)=0$ showing that the degree
(in its right-hand side presentation) is $0$.
\end{proof}

}


\def\cprime{$'$}

\end{document}